\documentclass[11pt,a4paper,twoside]{amsart}

\usepackage{amsmath,amssymb,amsthm}
\usepackage{inputenc}
\usepackage{tikz} 
\usetikzlibrary{arrows,shapes}
\usepackage{xifthen}
\usepackage{float}
\floatstyle{boxed} 
\usepackage{graphicx,subfigure}
\usepackage{caption}
\usepackage[T1]{fontenc} 
\usepackage{enumitem}
\usepackage{float}
\usepackage{mathtools}
 \usepackage{relsize}
\usepackage{mathrsfs}
\usepackage{pgfplots}
\usepackage{xcolor}
\usepackage[colorlinks=true, linkcolor=red, citecolor=blue, urlcolor=blue,hyperfootnotes=false]{hyperref}
\usepackage{cleveref}
\usepackage{etoolbox}
\usepackage{upgreek}

\def\@begintheorem#1#2{\par\bgroup{\sc #1 \ #2. }  \it \\\ignorespace }
\def\@opargbegintheorem#1#2#3{\par\bgroup{\sc #1\ #2 \ (#3).}  \it  \ignorespace}
\def\@endtheorem{\egroup}

\theoremstyle{plain}
\newtheorem{theorem}{Theorem}[section]

\theoremstyle{definition}
\newtheorem{definition}[theorem]{Definition}

\theoremstyle{remark}
\newtheorem{remark}[theorem]{Remark}

\theoremstyle{plain}

\theoremstyle{plain}
\newtheorem{lemma}[theorem]{Lemma}

\theoremstyle{plain}

\theoremstyle{plain}
\newtheorem{proposition}[theorem]{Proposition}

\numberwithin{equation}{section}

\newcommand{\C}{{\mathbb C}}

\newcommand{\D}{\mathcal{D}}
\newcommand{\Z}{{\mathbb Z}}

\newcommand{\R}{{\mathbb R}}

\newcommand{\K}{\mathcal{K}}

\renewcommand{\H}{\mathcal{H}}

\newcommand{\de}{\partial}

\newcommand{\supp}{\operatorname{supp}}

\newcommand {\e}{\`{e}~}

\DeclareMathOperator{\Realpart}{Re}
\renewcommand{\Re}{\Realpart}

{\left\lbrace\begin{array}{@{}l@{}}}%
{\end{array}\right.}

\DeclarePairedDelimiter{\abs}{\lvert}{\rvert}
\DeclarePairedDelimiter{\norm}{\lVert}{\rVert}
\DeclarePairedDelimiter{\seq}{\lbrace}{\rbrace}

\newcommand{\sign}{\mathop{\textrm{sign}}}

\DeclareMathOperator{\Ker}{Ker}
\DeclareMathOperator{\Ran}{Ran}

\newcommand{\bvec}[1]{\boldsymbol{#1}}
\newcommand{\id}{\boldsymbol{1}}
\newcommand{\di}{\,\mathrm{d}}

 \newcommand{\Span}{\operatorname{span}}

\newcommand{\cK}{\mathcal{K}}
\newcommand{\cN}{\mathcal{N}}

\newcommand{\inprodC}[2]{\left\langle #1 , #2\right\rangle_{\C^2}}
\newcommand{\inprod}[2]{\left \langle #1 , #2\right \rangle}

\newcommand{\dotprod}[2]{
  \bvec{#1}\mkern1mu{\cdot}\mkern1mu\bvec{#2} \,}
\newcommand{\sn}{ \bvec \sigma \cdot \bvec n}
\newcommand{\bdry}{\Sigma}

\title[Self-Adjointness of Dirac operators on corner domains]{Self-Adjointness of two dimensional Dirac operators\\ on corner domains}
\date{\today}
\author[F.~Pizzichillo]{Fabio Pizzichillo}
\address{F.~Pizzichillo, CNRS \& CEREMADE, Universit\'e Paris-Dauphine, PSL Research University, F-75016 Paris, France}
\email{pizzichillo@ceremade.dauphine.fr}

\author[H.~Van Den Bosch]{Hanne Van Den Bosch}
\address{H.~Van Den Bosch, Departamento de Ingenier\'{i}a Matem\'{a}tica \& CMM - Centro de Modelamiento Matem\'{a}tico, Universidad de Chile \& UMI--CNRS 2807, Beaucheff 851, Santiago, Chile}
\email{hvdbosch@cmm.uchile.cl}


\begin{document}
\begin{abstract}
We investigate the self-adjointness of the two-dimensional Dirac operator $D$,  with 
\emph{quantum-dot} and \emph{Lorentz-scalar $\delta$-shell}
boundary conditions, on piecewise $C^2$ domains with finitely many corners.
For both models, we prove the existence of a unique self-adjoint realization whose domain is included in the Sobolev space $H^{1/2}$, the formal form domain of the free Dirac operator.
The main part of our paper consists of a description of the domain of $D^*$ in terms of the domain of $D$ and the set of harmonic functions that verify some \emph{mixed} boundary conditions.
Then, we give a detailed study of the problem on an infinite sector, where explicit computations can be made: we find the self-adjoint extensions for this case.
The result is then translated to general domains by a coordinate transformation.
\end{abstract}
\maketitle

\section{Introduction}
In this paper we study the self-adjoint realizations of the two-dimensional Dirac operator with boundary conditions on corner domains.

The free massless Dirac operator in $\R^2$ is given by the differential expression
\begin{equation} \label{eq:def-Dirac}
H:=-i\begin{pmatrix}
0 & \partial_x-i \partial_y \\
\partial_x + i \partial_y& 0
\end{pmatrix} = -i \bvec{\sigma} \cdot \nabla,
\end{equation}
where $\bvec \sigma=(\sigma_1,\sigma_2)$ and the \emph{Pauli matrices} are defined as
\[
\quad{\sigma}_1 =\left(
\begin{array}{cc}
0 & 1\\
1 & 0
\end{array}\right),\quad {\sigma}_2=\left(
\begin{array}{cc}
0 & -i\\
i & 0
\end{array}
\right),\quad{\sigma}_3=\left(
\begin{array}{cc}
1 & 0\\
0 & -1
\end{array}\right).
\]
The Dirac operator describes the evolution of a relativistic particle with spin $\tfrac12$.
It also arises as an effective description of electronic excitations in materials with a hexagonal lattice structure, such as graphene.
The free operator in $\R^2$ can be seen to be self-adjoint on $\D(H):=H^1(\R^2,\C^2)$, since it is equivalent to multiplication by $\bvec \sigma \cdot \bvec k$ after Fourier transform.
For more details, see for instance \cite{thaller}.

Let $\Omega \subset \R^2$ be a connected domain with $\bdry:=\partial\Omega$. Throughout, $\gamma$ is the trace at $\bdry$.
We denote by $\bvec n$ the outward normal and by $\bvec t$ the tangent vector to $\bdry$ chosen in such a way that $(\bvec n, \bvec t)$ is positively oriented. 
In this paper, we will study two perturbations of the free Dirac operator related to the domain $\Omega$.

The \emph{quantum-dot} operator arises when the Dirac fermions are confined by a termination of the lattice or by some type of potential. 
The best-known example of these boundary conditions is the one known in different communities as infinite mass, armchair, MIT-bag or chiral, as introduced in \cite{BerryMondragon} for theoretical reasons, or experimentally studied, for instance in \cite{PonomarenkoetAl}. In \cite{stockmeyerinfinite}, it was shown that this operator is the limit (in a suitable sense) of the free Dirac operators perturbed by a mass term localized outside the domain $\Omega$, when this mass tends to infinity.  
The quantum-dot operator $D^Q$ acts as $H$ on 
the domain
\begin{equation}\label{eq:def-dom-DQ}
\D (D^Q) = \{u \in H^1(\Omega,\C^2): P^Q_\eta  \gamma u= 0  \}.
\end{equation}
where  the boundary condition $P^Q_\eta$ is parametrized by $\eta \in [0,\pi)$, and it is given by 
\begin{equation}\label{eq:Q-dot-BC}
P^Q_\eta \gamma u:= \frac{1}{2}(1 -A^Q_\eta)\gamma u,  \quad A^Q_\eta := \sin (\eta ) \bvec\sigma \cdot \bvec t(s) + \cos( \eta) \sigma_3. 
\end{equation}
Throughout this paper we assume that $\eta\in (0,\pi)$. The case $\eta=0$ is known as \emph{zig-zag boundary value conditions}. It is, mathematically speaking, very different from the other cases and we plan to study corners in this model in the future.

The \emph{$\delta$-shell interaction} arises as a limiting case of the free Dirac operator perturbed by a potential that is strongly localized on the curve $\bdry$.
Formally, one can think of this perturbation as a potential that is a coupling constant times the Dirac $\delta $ distribution on $\bdry$.
In order to make sense of this mathematically, it has to be considered as a boundary value problem.
The action of the $\delta$-shell operator on a function defined on the whole space
can be seen as the direct sum of the action of the free Hamiltonian on the restriction of the function on $\Omega$ and its complement $\Omega^c$. Along the curve $\bdry$, both functions are linked by a special type of transmission condition given in \eqref{eq:shell-BC}.
In dimension three, in \cite{klein, sphericalnote} it is shown that this type of operator is exactly the limit of the operators with smooth potentials that approximate a delta function on the surface.
The case that we study here, is the case where this potential takes the form of position-dependent mass term, or formally $ m \delta_{\bdry} \sigma_3$.
We call this model the \emph{Lorentz-scalar} $\delta$-shell (as opposed to an electrostatic delta-shell generated by $ V \delta_{\bdry} \id_{\C^2}$), since
it is invariant under Lorentz transformations.
We study the \emph{Lorentz-scalar $\delta$-shell} operator $D^L$, defined as the action of $H$ on pairs of spinors $u_+, u_-$ defined in $\Omega^+ \equiv \Omega$ and $\Omega^- \equiv\R^2 \setminus \overline \Omega$, with domain
\begin{equation}\label{eq:def-dom-DL}
\D(D^L) = \left\{(u_+, u_-) \in H^1(\Omega^+,\C^2) \times H^1(\Omega^-,\C^2): 
P^L_\mu  (\gamma u_+, \gamma u_-)= 0  \right
\},\\
\end{equation}
where  the boundary condition $P^L_\mu $ is parametrized by $\mu \in (-1,1)$, and defined as the orthogonal projection on pairs of spinors satisfying
\begin{equation} \label{eq:shell-BC}
M^+_\mu \gamma u_++ M^-_\mu \gamma u_- =0,
\quad
M^\pm_\mu:=\left(\pm i \bvec \sigma \cdot \bvec n + \mu \sigma_3 \right).
\end{equation}
For the physical interpretation, $2 \mu$ is the \emph{mass} of the $\delta$ shell. Throughout this paper, we assume that $\mu\neq 0$ since the case $\mu=0$ coincides with the free Dirac operator on $\R^2$.

When $\Omega$ is a $C^2$ domain, both operators are self-adjoint.
In other words, the boundary value problem has an elliptic regularity property.
For the quantum-dot model, this was shown in \cite{benguria2017self}.
The $\delta$-shell interaction has been studied previously in dimension three, but the 3-dimensional theory also applies in dimension 2.
Self-adjointness for 3-dimensional $\delta$-shell interactions has been obtained in \cite{ditexnseb, amv1,amv2,amv3, ourmieres2018strategy, behrndt2016dirac, behrndt2018spectral, holzmann2018dirac}, in increasingly general settings, and  we refer to \cite{diracdeltareview} for a review on the topic.

In this paper we are interested in relaxing the smoothness hypothesis on the domain: we consider domains with \emph{corners}. This is justified by several reasons: first the fact that for numerical approximation, smooth curves are approximated by polygons.
Secondly, from a mathematical point of view this turns out to be an interesting question, and it goes beyond a mere generalization of the methods in previous works.
Indeed, if we compare the same problem for the Schr\"odinger operator, we obtain that for convex corners the operator admits a one-parameter family of self-adjoint extensions. Any element in this set is the norm resolvent limit of a suitable sequence of Friedrichs-Dirichlet Laplacians with point interactions, see \cite{posilicano2013many}.
Although we can expect the existence of a family of self-adjoint extensions, they cannot correspond to point interactions,
since the point interaction for the Dirac operator is not well defined in dimension greater than one.

To our knowledge, boundary value problems on corner domains for the $2$-d Dirac operator have been treated in only two works.
In \cite{le2018self}, the case of polygons has been treated for the \emph{MIT}-bag model, a particular case of quantum-dot boundary conditions.
In the case where $\Omega$ is a sector, the authors prove that the operator defined on $H^1$ is self-adjoint for opening angles in $(0,\pi)$ and it is not self-adjoint for opening angles in $( \pi, 2 \pi)$. In the latter case, it admits a one-parameter family of self-adjoint extensions and among them, only one has domain included in the Sobolev space $H^{1/2}$.
In \cite{cassano2019self}, the authors study the case of two-valley Dirac operator on a wedge in $\R^2$ with infinite mass boundary conditions, whith an additional sign flip at the vertex. They parametrize all its self-adjoint extensions, proving that there exists no self-adjoint extension, which can be decomposed into an orthogonal sum of two two-component operators. This property is related to the \emph{valley-mixing} effect.

These two papers strongly depend on the radial symmetry of the domain.
We generalize the results in \cite{cassano2019self,le2018self} to more general boundary conditions and to \emph{curvilinear polygons}.
Our main result, \Cref{thm:main}, states that for a general bounded and piecewise $C^2$-regular domain $\Omega$ with finitely many corners,  the operators $D^L$ and $D^Q$ have a unique self-adjoint extensions with domain contained in $H^{1/2}$, the natural form domain of $H$. Since functions in $H^{1/2}$ do not necessarily have boundary traces, we need to introduce some definitions before stating this precisely.
The proofs in \cite{le2018self,cassano2019self} use an exact decomposition of the operator on the wedge in angular momentum subspaces. 
This strategy could also work for the operator under consideration here. However, we chose a different method that can be seen as a Dirac analogue of the tools developed in \cite{grisvard} for the case of second-order elliptic operators on corner domains. Indeed, in \Cref{prop:D(D^*)}, we characterize the domain of the adjoint operator in terms of the operator defined on $H^1$ plus
the set of $\C^2$-valued harmonic functions that verify some \emph{mixed} boundary condition, see \Cref{prop:D(D^*)} for more details.
This fact holds independently of details about the domain and may generalize to other or the three-dimensional case. The bottomline is that one has to obtain information about harmonic spinors near corners in order to completely solve the problem.

Before stating our main theorem, following \cite{grisvard}, we define precisely the class of domains under consideration.
\begin{definition}
Let $\Omega\subset\R^2$ be a bounded and simply-connected domain and let $\bdry=\de\Omega$.
We say that $\Omega$ is a \emph{curvilinear polygon} of class $\C^2$ if and only if 
if for every $x\in\bdry$ there exists a neighbourhood $V$ of $x$ in $\R^2$ and a mapping $\psi:V\to\R^2$  such that
\begin{enumerate}[label=\emph{(\roman*)}]
\item $\psi$ is injective;
\item $\psi$ and $\psi^{-1}$ (defined on $\psi(V)$) are of class $C^2$;
\item denoting with $\psi_j$ the $j$-th component of $\psi$,  $\Omega\cap V$ is 
\begin{enumerate}[label=\emph{(\alph*)}]
\label{item:a-def.CurvPol}
\item either $\left\{ y\in V:\psi_2(y)<0\right\}$,
\item either $\left\{ y\in V:\psi_1(y)<0\ \text{and}\ \psi_2(y)<0\right\}$;
\label{item:b-def.CurvPol}
\item or $\left\{ y\in \Omega:\psi_1(y)<0\ \text{or}\ \psi_2(y)<0\right\}$.
\label{item:c-def.CurvPol}
\end{enumerate}
\end{enumerate}
For $x\in\bdry$, we say that $x$ is a \emph{convex corner} in case \ref{item:b-def.CurvPol} and a \emph{non-convex corner} in case  \ref{item:c-def.CurvPol}.
\end{definition}

We will use lowercase letters like $u$, $v$, \dots , to refer to spinors in $L^2(\Omega, \C^2)$ or pairs of spinors in $L^2(\Omega^+, \C^2) \times L^2(\Omega^-, \C^2) $.
When we have to distinguish components,
\[
u = \begin{pmatrix} u_1 \\ u_2 \end{pmatrix}, \quad u_i \in  L^2(\Omega, \C)
\]
in the first case, and
\[
u = \left(u_+, u_-\right) =  \left( \begin{pmatrix} u_{+,1} \\ u_{+,2} \end{pmatrix}, \begin{pmatrix} u_{-,1} \\ u_{-,2} \end{pmatrix} \right), \quad u_{\pm,i} \in  L^2(\Omega^\pm, \C)
\]
in the second case.

From time to time, we omit the superscripts $Q$ and $L$ for statements that apply to both $D^Q$ and $D^L$. 
We define the maximal domain of $H$
for a domain $\mathcal{O}$, by
\[
\cK (\mathcal{O}) := \{ u \in L^2(\mathcal{O},\C^2): H u \in L^2(\mathcal{O},\C^2) \}. 
\]
Since $\mathscr{D}(\Omega, \C^2) \subset \D( D^Q) $, the adjoint operator $(D^Q)^*$ acts as $H$ and $\D((D^Q)^*)\subset\cK(\Omega)$. Analogously $\D((D^L)^*)\subset\cK(\Omega^+)\times\cK(\Omega^-)$.
We can now state our main result.

\begin{theorem} \label{thm:main}
Let $\Omega$ be a bounded and simply connected curvilinear polygon of class $C^2$.
Let the operators $D^Q$ and $D^L$ be defined respectively as in \eqref{eq:def-dom-DQ} and \eqref{eq:def-dom-DL}.
The operators $D^Q$ and $D^L$ admit self-adjoint extensions $D_0^Q$ and $D_0^L$
with domains 
\[
\begin{split}
\D(D_0^Q)&:=\seq{u\in H^{1/2}(\Omega,\C^2)\cap \cK(\Omega): \ P^Q_\eta\gamma u=0};\\
\D(D^L_0) &:=
\left\{
\begin{array}{l}
u=(u_+, u_-) \in \left( H^{1/2}(\Omega^+,\C^2)\cap \cK(\Omega^+) \right)\times \left( H^{1/2}(\Omega^-,\C^2)\cap \cK(\Omega^-)\right) :
\\
 P^L_\mu  (\gamma u_+, \gamma u_-)= 0  
\end{array} 
 \right\}.
\end{split}
\]
\end{theorem}
\begin{remark}
The functions in $\cK (\Omega^\pm)$  do not have $H^1$ regularity and so,
a priori, the boundary conditions are ill-defined. 
Nevertheless, we will see in \Cref{lem:trace-ext} that the boundary trace can be defined in a weaker sense. Also, away from the corners, elements of $\D(D_0^Q)$
are $H^1$ and thus the boundary conditions hold in the usual sense in any subset of $\Sigma$ not containing corners.
\end{remark}

Theorem~\ref{thm:main} will be a consequence of a more general result about the decomposition of the domains of the adjoint operators $(D^Q)^*$ and $(D^L)^*$. 
We first define the localized operators close to each corner and the spaces of solutions of the corresponding adjoint problems.

\begin{definition}\label{def:D^*_rho}
Let $\Omega$ be a curvilinear polygon of class $C^2$ and let $D$ be defined as in \eqref{eq:def-dom-DQ} and \eqref{eq:def-dom-DL}.
Let $\mathcal{C}$ be the finite set of the corners of $\Sigma$. 
For every $c\in \mathcal{C} $ we define
\begin{equation}\label{eq:def.K_rho}
\begin{split}
N^Q_\rho(c)&:=
\{ u \in \cK (\Omega\cap B(c,\rho)) : 
\Delta u=0\ \text{and}\ P^Q_\eta\gamma u=P^Q_\eta\gamma (Hu)=0\ \text{on}\ \Sigma\cap   B(c,\rho)
\};
\\
N^L_\rho(c_j)&:=
\left\{
\begin{array}{l}
u=(u^+,u^-)\in  \cK (\Omega^+\cap B(c,\rho))\times \cK (\Omega^-\cap B(c,\rho)):\\
\Delta u=(\Delta u^+,\Delta u^-)=0\  \text{and}\ P^L_\mu\gamma u=P^L_\eta\gamma (Hu)=0\ \text{on}\ \Sigma\cap   B(c,\rho)
\end{array}
\right\}.
\end{split}
\end{equation}
So, the spaces $N_\rho(c)$ contain harmonic functions in a neighborhood of the corner $c$ satisfying some mixed boundary conditions.
We will also need to extend these functions to the entire domain. To this end, fix a radial cut-off $\phi$ such that
\begin{equation}\label{eq:def.cutoff}
\phi \in C^{\infty}(\R^2, [0,1]),\quad\text{and} \quad
\phi(x)=
\begin{cases}
1&\text{for } |x|<1/3;\\
0&\text{for } |x|>2/3;
\end{cases}.
\end{equation}
We define
\begin{equation}\label{eq:def.N_rho}
\mathcal{N}_{\rho}(c):=
\left\{\phi\left(\tfrac{x-c}{\rho}\right) u^E (x) :  u  \in N_\rho(c) \right\},
\end{equation}
where, for $u$ defined in $B(c, \rho)$, $u^E$ denotes its extension by zero.
\end{definition}
\begin{remark}
Since $H^2$ acts as $- \Delta$, if $u\in N_\rho(c)$ then  $Hu \in \cK$ and the boundary trace of $Hu$  is well-defined in the generalized sense.
\end{remark}
With these definitions, we can state the following theorem.

\begin{theorem}\label{prop:D(D^*)}
Let $\Omega$ be a curvilinear polygon of class $C^2$ with finitely many corners. Denote by $\mathcal{C}$ be the set of its corners.
Let $D$ be defined as in \eqref{eq:def-dom-DQ} and \eqref{eq:def-dom-DL} . For $\rho>0$ define $\mathcal{N}_\rho(c)$ for all $c\in\mathcal{C}$ as in \eqref{eq:def.N_rho}.
Then, we have a decomposition:
\begin{equation}\label{eq:D(D^*)}
\D(D^*) = \D(D)+\sum_{c\in\mathcal{C}}\mathcal{N}_\rho(c)
\end{equation}
\end{theorem}

We will prove Theorem~\ref{prop:D(D^*)} in Section~\ref{sec:general}.
In Section~\ref{sec:separation}, we use separation of variables to compute a basis of $\mathcal{N}_\rho(c)$ in the case of a wedge with straight edges.
This allows to obtain the complete description of self-adjoint extensions for corners with straight edges.
In section~\ref{sec:curvilinear}, we obtain a unique self-adjoint extension with domain $\D(D^*) \cap H^{1/2}$ for curvilinear polygons.

\section{General considerations and proof of \texorpdfstring{\Cref{prop:D(D^*)}}{Theorem 1.3}}\label{sec:general}

In this section, we group some properties of the operators $D^Q$ and $D^L$, their adjoints, and finally prove Theorem~\ref{prop:D(D^*)} . 
We assume that $\Omega$ is a bounded curvilinear polygon of class $C^2$. 

We start by some identities that are well-known from the smooth case.
In order to simplify the computations, we rewrite the boundary condition $P^Q_\eta$ and $P^L_\mu$ defined respectively in \eqref{eq:Q-dot-BC} and \eqref{eq:shell-BC}. 
Throughout the paper, we use the canonical identification $\R^2\sim \C$, that is for any $\bvec x\in \R^2$, we will denote $x:=\bvec{x}_1 + i \bvec{x_2}\in\C$. In particular, with this notation $n=\bvec{n}_1+i \bvec{n}_2$ and $t=\bvec{t}_1+ i \bvec{t}_2$, where $\bvec{n}$ is the outward unit normal and $\bvec{t} = (-\bvec n_2, \bvec n_1)$ is the tangent vector with our choice of orientation.

For the quantum-dot model, 
$P^Q_\eta \begin{pmatrix}
\gamma u_1 \\ \gamma u_2
\end{pmatrix} = 0$ if and only if 
\[
\gamma u_2 = B t \gamma u_1, \quad\text{where}\ B := \frac{\sin \eta}{1 - \cos \eta}.
\]
So we can use equivalently
\begin{equation}\label{eq:def-dom-DQ-R}
\D(D^Q) =\seq{u\in H^1(\Omega,\C^2): \gamma u_2=B t \gamma u_1}.
\end{equation}
From \eqref{eq:def-dom-DQ-R}, the operator $D^Q$ depends on a parameter $B>0$. We use the notation $D^{Q,B}$ to stress this dependence.
Setting 
\[
M_B:=\begin{pmatrix}
B^{-1/2}&0\\
0& B^{1/2}
\end{pmatrix},
\]
 we have that
\[
D^{Q,B}=M_B D^{Q,1} M_B.
\]
Thanks to this and since the matrix  $M_B$ is Hermitian and invertible,
the problem of the self-adjointness for the operator $D^{Q,B}$ is equivalent to the problem of self-adjointness for the operator $D^{Q,1}$. For this reason, from now on we we only assume that $B=1$ or equivalently $\eta=\pi/2$.  This kind of boundary condition, is called \emph{infinite-mass} boundary condition.
Finally, for sake of clarity we identify $P^Q=P^Q_{\pi/2}$.

For the Lorentz-scalar $\delta$-shell, we have that $P^L_\eta \gamma(u_+,u_-) = 0$ if and only if  
\[
\gamma u_- =   - (M_\mu^-)^{-1} M_\mu^+ \gamma u_+ = - \cosh(\alpha) \gamma u_+  - \sinh(\alpha) \dotprod{\sigma}{t} \gamma u_+,\quad\text{where}\ \tanh(\alpha) = \frac{2 \mu}{1+\mu^2}.
\]
Again, we will mainly use this characterization of the domain
\begin{equation} \label{eq:def-dom-DL-R}
\D (D^L) =
\left\{
 (u_+, u_-)  \in H^1(\Omega_+,\C^2 )\times H^1(\Omega_+,\C^2 ):
 \gamma u_- = (\cosh(\alpha) - \sinh(\alpha) \dotprod{\sigma}{t}) \gamma u_+ 
\right\}.
\end{equation}
We now list some useful identities. For smooth $u,v$, these identities follow from the divergence theorem and identities of the Pauli matrices. They follow for general $u,v$ by an approximation argument that requires some extra care in the case of limited boundary regularity. We provide a detailed proof in Appendix~\ref{sec:appendix}.
\begin{lemma} \label{lem:ibp} Let 
$\mathcal{O}$ be a piecewise $C^1$ domain,  $\bvec{n}_{\mathcal{O}} $ the outward normal. Let $\Omega$ be a curvilinear polygon of class $C^2$ with boundary $\Sigma$.
We define, almost everywhere on $\Sigma$, $\kappa \equiv \dotprod{t}{}\partial_{\bvec t} \bvec n$, which equals, up to a sign depending on the orientation, the piecewise continuous curvature of the boundary.
 \begin{enumerate}[label=(\roman*)]
  \item For all $u, v \in H^1(\mathcal{O})$ , we have
  \begin{equation} \label{eq:ibp}
   \inprod{u}{Hv}_{L^2(\mathcal{O})} - \inprod{Hu}{v}_{L^2(\mathcal{O})}  =i\int_{\partial \mathcal{O}} \inprodC{\bvec \sigma \cdot \bvec n_\mathcal{O}u}{ v};
  \end{equation}
  \item For all $u, v \in H^1(\mathcal{O})$
    \begin{equation}    \label{eq:q-form-general}
    \inprod{Hu}{Hv}_{L^2(\mathcal{O})}  = \inprod{\nabla u}{\nabla v}_{L^2(\mathcal{O})}  + \int_{\partial \mathcal{O}} \inprodC{u}{i \sigma_3 \partial_{\bvec t_\mathcal{O}} v};
  \end{equation}
     \item For all $u, v \in \D(D^Q)$
  \begin{equation} \label{eq:q-form-D-Q}
  \inprod{D^Q u}{ D^Q v}_{L^2(\Omega)} = \inprod{\nabla u}{\nabla v}_{L^2(\Omega)} + \int_{\bdry} \frac{\kappa}{2} \inprodC{u}{v};
  \end{equation}
  \item For all $u,v \in \D(D^L)$
  \begin{equation} \label{eq:q-form-D-L}
  \begin{split}
   \inprod{D^L u}{D^L v}_{L^2(\Omega^+)\times L^2( \Omega^-)} =& \inprod{\nabla u}{\nabla v}_{L^2(\Omega^+)\times L^2( \Omega^-)} - \sinh(\alpha)\int_{\bdry} \kappa  \inprodC{u_-} {\dotprod{\sigma}{t} v_+} 
 \end{split}
  \end{equation} 
 \end{enumerate}
\end{lemma}

With these identities, we check that the operators defined previously are symmetric.
\begin{proposition}\label{prop:symmetry}
The operators $D^Q$ and $D^L$, defined in \eqref{eq:def-dom-DQ} and \eqref{eq:def-dom-DL} respectively, are symmetric and closed.
\end{proposition} 

\begin{proof}
Tanks to \eqref{eq:ibp} we have that $D^Q$ is symmetric if 
\[
\int_\bdry \inprodC{\sn_\bdry u}{v}=0,
\quad \text{for all}\  u,v\in D^Q. 
\]
Let $u,v \in\D(D^Q)$, then $P^Q\gamma u=P^Q\gamma v=0$, where $P^Q$ is defined in \eqref{eq:Q-dot-BC} for $\eta=\pi/2$. 
Moreover, since $\bvec{\sigma} \cdot \bvec{n}_\bdry$ anti-commutes with both $\sigma_3$ and $\bvec{\sigma} \cdot \bvec{t}_\bdry$, it anti-commutes with $A^Q$ and so
\[
P^Q\sn_\bdry =\sn_\bdry(1-P^Q).
\]
Thanks to this, since $P^Q$ is a hermitian matrix on $\C^2$, we can conclude that 
\[
0=\inprodC{P^Q \gamma u}{\sn_\bdry \gamma v}
=\inprodC{\sn_\bdry\gamma u}{ (1-P^Q)\gamma v}=
\inprodC{\sn_\bdry\gamma u}{ \gamma v}.
\]
Thus $D^Q$ is symmetric.

Let us analyse  $D^L$. Thanks to \eqref{eq:ibp} we have that $D^L$ is symmetric if and only if
\[
\int_\bdry 
\inprodC{\sn_\bdry u_+}{v_+}
-
\inprodC{\sn_\bdry u_-
}{v_-}=0,\quad
\text{for all}\ u, v 
\in\D(D^L).
\]
Since $\sn_\bdry$ anti-commutes with $\sigma_3$ we have that 
\[
\begin{array}{ccc}
(M^\pm_\mu)^*=M^\mp_\mu, &
(M^\pm_\mu)^{-1}=(\mu^2-1)^{-1}M^\pm_\mu, &
M^\pm_\mu\sn_\bdry=-\sn_\bdry M^\mp_\mu.
\end{array}
\]
Using these properties, the boundary condition can be rewritten as
\[  u_- (x) = - (M_\mu^-)^{-1} M_\mu^+ u_+(x) = (1-\mu^2)^{-1} M_\mu^- M_\mu^+ u_+(x) , \quad x \in \bdry\] and the same holds for $v$. Thus, we compute 
\begin{align*}
\inprodC{ u_-}{\sn_\bdry v_-} 
&= (1-\mu^2)^{-2} \inprodC{M_\mu^- M_\mu^+  u_+}{\sn_\bdry M_\mu^- M_\mu^+ v_+} \\
&= (1-\mu^2)^{-2} \inprodC{  u_+}{M_\mu^- M_\mu^+\sn_\bdry M_\mu^- M_\mu^+ v_+} \\
&= (1-\mu^2)^{-2} \inprodC{  u_+}{\sn_\bdry M_\mu^+ M_\mu^- M_\mu^- M_\mu^+ v_+} =  \inprodC{ u_+}{\sn_\bdry v_+} 
\end{align*}
Therefore, the boundary term vanishes and $D^L$ is symmetric.

Finally, to obtain the closedness of $D^Q$, we start from \eqref{eq:q-form-D-Q}.
There exists a constant $C_\bdry>0$, depending only on the curvature of $\bdry$, such that
\begin{align*}
\norm{u}_{H^1(\Omega )}^2 
&\le   \norm{D^Q u}_{L^2(\Omega)}^2 + \norm{u}_{L^2(\Omega)}^2 +  C_\bdry \norm{\gamma u}_{L^2(\bdry)}^2 \\
& \le \norm{D^Q u}_{L^2(\Omega)}^2 + \norm{u}_{L^2(\Omega)}^2 +  C_\bdry  \norm{ u}_{H^{1}( \Omega)} \norm{u}_{L^2( \Omega)} \\
 &\le \norm{D^Q u}_{L^2(\Omega)}^2 + (1 + \epsilon^{-1} C_\bdry)\norm{u}_{L^2(\Omega)}^2 +  \epsilon C_\bdry \norm{u}_{H^1(\Omega )}^2.
\end{align*}
Thus, taking $\epsilon$ sufficiently small, we can find a constant such that 
\begin{equation}\label{eq:controlH1}
\norm{u}_{H^1(\Omega)} \le C \left( \norm{D^Q u}_{L^2(\Omega)} + \norm{ u}_{L^2(\Omega )} \right) .
\end{equation}
Let $(u_n)_n\subset \D(D^Q)$ 
be a Cauchy sequence in the graph norm for $D^Q$.
 Then \eqref{eq:controlH1} implies that $(u_n)_n$ is a Cauchy sequence in $H^1(\Omega,\C^2)$ and so  there exists $u\in H^1(\Omega,\C^2)$ such that 
 $(u_n)\to u$ in $H^1$.
Since the boundary trace map $\gamma$ is continuous from $\D(D^Q)$ to $H^{1/2}(\bdry)$ , we have that $\gamma u_n\to \gamma u$, and so $\gamma u$ verifies the quantum-dot boundary conditions.
Thus, $D^Q$ is closed.

The proof for $D^L$ is analogous. We start this time from \eqref{eq:q-form-D-L} to conclude that there exists a constant such that there exists $C_\bdry>0$ only depending on $\bdry$ such that
to obtain 
\begin{align*}
 \norm{u}_{H^1(\Omega^+ )\times H^1(\Omega^-)}^2 
\le&   \norm{D^L u}_{L^2(\Omega^+) \times L^2(\Omega^-)}^2 + \norm{u}_{L^2(\Omega^+) \times L^2(\Omega^-)}^2+  C_\bdry \norm{\gamma u}_{L^2(\bdry)}^2 
\end{align*}
This expression can be bounded in a completely analogous way to obtain
\begin{equation}\label{eq:controlH1-L}
\norm{u_+}_{H^1(\Omega^+)}+\norm{u_-}_{H^1(\Omega^-)}  \le C \left( \norm{D^L u} + \norm{ u} \right) .
\end{equation}
Reasoning as before, we conclude that $D^L$ is closed.
\end{proof}

Now, we move on to study $\D(D^*)$. Since test functions are included in the domain of $D$, its adjoint $D^*$ acts, in distribution sense, as the differential expression $H$.
Therefore, the domain of the adjoint is included in the maximal domain of the elliptic differential expression $H$. Spinors in the maximal domain have boundary traces, as is the case for functions in the maximal domain of second order elliptic operators, see e.g, \cite[Sec 1.5.3]{grisvard}.
Furthermore, functions in the domain of the adjoint satisfy boundary conditions in a weak sense.

\begin{lemma} \label{lem:trace-ext}
Let $\mathcal{O}$ be a a curvilinear polygon of class $C^1$. The map $\dotprod{\sigma}{n} \gamma : H^1(\mathcal{O},\C^2) \to L^2(\partial \mathcal{O},\C^2)$ extends to a bounded map $T : \cK(\mathcal{O}) \to H^{-1/2}(\partial \mathcal{O},\C^2)$. 
\end{lemma}
We defer the proof of \ref{lem:trace-ext} to \Cref{sec:appendix}. We are ready now to give a first characterization of $\D(D^*)$.
\begin{proposition}\label{prop:def-D(D^*)}
Let $D^Q$ be the quantum-dot operator defined as in \eqref{eq:def-dom-DQ}. Then
\begin{equation}
\label{eq:def-dom-DQ^*}
\D((D^Q)^*) = \{u \in \cK(\Omega) : P^Q_\eta 
\dotprod{\sigma }{n} 
T u = 0\},
\end{equation}
where the boundary conditions hold in the sense that, for all $f \in H^{1/2}(\bdry,\C^2)$ such that $P^Q_\eta f \in H^{1/2}(\bdry,\C^2)$, we have 
\[
T u[(1-P^Q_\eta)  f] = 0.
\]
Let $D^L$ be the Lorentz-scalar $\delta$-shell operator defined as in \eqref{eq:def-dom-DL}. Then
\begin{equation}
\label{eq:def-dom-DL^*}
\D((D^L)^*) = \{ u=(u_+,u_-) \in \cK (\Omega^+ ) \times \cK (\Omega^- ) : P^L_\mu
 T u =0\},
\end{equation}
where the boundary conditions hold in the sense that, for all $f_+$ and $f_-$ in $  H^{1/2}(\bdry,\C^2)$ such that $M^\pm_\mu f_\pm \in  H^{1/2}(\bdry,\C^2)$, we have 
\[
T u_+[M^-_\mu f^+]+ Tu_-[M^+_\mu f_-]=0. 
\]
\end{proposition}
\begin{remark}
If $v\in \D((D)^*)$ is supported away from the corners, multiplication of $Tv$ by $\dotprod{\bvec \sigma }{n}$ and $P^Q_\eta$ makes sense and the boundary conditions hold in the usual sense.
\end{remark}
\begin{remark}
If $\mathcal C$ is the set of corner points of $\partial \mathcal{O}$, if $f$ and $P^Q_\eta f$ are in $H^{1/2}(\partial \mathcal{O})$, $f$ vanishes on $\mathcal C$ in the sense that it can be written as a $H^{1/2}$ limit of functions with compact support in $\partial \mathcal{O} \setminus \mathcal C$. 
\end{remark}

\begin{proof}[Proof of \Cref{prop:def-D(D^*)}.]
Since $H_0^1(\Omega, \C^2) \subset \D( D^Q) $, for $v \in  \D( (D^Q)^*)$ we have
  \[
  (D^Q)^* v = H v,
  \]
  which implies that  $\D( (D^Q)^*) \subset \cK(\Omega)$.
  
Set 
\[
\D^*:= \{u \in \cK(\Omega) : P^Q_\eta \dotprod{\sigma }{n} T u = 0\},
\]
and let us prove first that $\D^* \subset \D((D^Q)^*)$.
  Let $v \in \D^*$, then for all $u \in \D(D^Q)$, $\gamma u = (1-P^Q_\eta)\gamma u \in H^{1/2}(\bdry) $.
  By \eqref{eq:trace-ext-def}, 
  \[
\inprod{D^Q u}{v}_{L^2(\Omega)} - \inprod{u}{H v}_{L^2(\Omega)}  = 0,  
  \]
  so $v \in \D( (D^Q)^*)$.
  
  For the opposite inclusion, take $v \in \D((D^Q)^*)$ and $f \in H^{1/2}(\bdry,\C^2)$ such that $P^Qf \in H^{1/2}(\bdry,\C^2)$.
  Then $u \equiv \xi (1-P^Q) f \in \D(D^Q)$, where $\xi : H^{1/2}(\bdry,\C^2) \mapsto H^1(\Omega, \C^2)$ is a bounded extension operator, and therefore,
\[\begin{split}
T v [(1-P^Q_\eta) f ] = i \inprod{v}{D^Q u}_{L^2(\Omega)} - i \inprod{(D^Q)^* v}{ u}_{L^2(\Omega)} = 0. 
\end{split}
\]

The proof \eqref{eq:def-dom-DL^*} is completely analogous.
\end{proof}

The last ingredient for the proof of \Cref{prop:D(D^*)} is the following lemma.

\begin{lemma}  \label{lem:Fredholm}
For every $c\in\mathcal{C}$ let  $N_{\rho,c}$ be defined as in \eqref{eq:def.N_rho} and define
 $D_{\rho,c}^Q $ and $D_{\rho,c}^L$ as the action of $H$
on the domains
\begin{equation}\label{eq:def:D_rho,c}
\begin{split}
\D(D^Q_{\rho,c})& := \{ u \in H^1 (\Omega\cap B(c, \rho)) : u^E \in \D(D^Q) \};\\
 \D(D^L_{\rho,c})& :=
 \left\{
 u=(u^+,u^-)\in H^1 (\Omega^+\cap B(c, \rho),\C^2)\times H^1 (\Omega^-\cap B(c_i, \rho),\C^2): 
 u^E \in \D (D^L)
\right\};
 \end{split}
\end{equation} 
and let $D_{\rho,c}^*$ be its adjoint.
Then
\begin{enumerate}[label=(\roman*)]
\item \label{item:1-ker} $D_{\rho,c}$ is closed and symmetric;
\item \label{item:2-ker} $\Ran (D_{\rho,c})$ is closed and $\Ker(D_{\rho,c})=\{0\}$;
 \item \label{item:3-ker} $\Ker((D_{\rho,c}^*)^2)=N_\rho(c)$.
\end{enumerate} 
\end{lemma}

\begin{proof}
Let us analyse $D_{\rho,c}^Q$. For all $u \in \D(D_{\rho,c}^Q)$, the extension by zero $u^E$ is in $\D(D^Q)$.
Thanks to this it is easy to see that $D_{\rho,c}^Q$ is symmetric.
By applying \eqref{eq:controlH1} to $u^E$, we find that a constant such that 
\begin{equation}\label{eq:controlH1-loc-op}
\norm{u}_{H^1(\Omega\cap B(c,\rho))} \le C \left( \norm{D_{\rho,c}^Q u}_{L^2(\Omega\cap B(c,\rho))} + \norm{ u}_{L^2(\Omega \cap B(c,\rho))} \right) .
\end{equation}
Let $(u_n)_n\subset \D(D_{\rho,c}^Q)$ 
be a Cauchy sequence in the graph norm for $D_{\rho,c}^Q$.
 Then \eqref{eq:controlH1-loc-op} implies that $(u_n^E)_n$ is a Cauchy sequence in $H^1(\Omega,\C^2)$ and so  there exists $u\in H^1(\Omega,\C^2)$ such that 
 $(u_n^E)\to u$ in $H^1$ and $\supp (u) \subset \overline{B(c, \rho)}$.
Since the boundary trace map $\gamma$ is continuous from $\D(D_{\rho,c}^Q)$ to $H^{1/2}(\bdry)$ , we have that $\gamma u_n\to \gamma u$, and so $\gamma u$ verifies the quantum-dot boundary conditions.
Thus, $D_{\rho,c}^Q$ is closed.

Next, since $\D(D_{\rho,c}^Q)$ is compactly embedded in $L^2(\Omega,\C^2)$ and $D_{\rho,c}^Q$ is closed, thanks to \eqref{eq:controlH1-loc-op}, and the Peetre characterization theorem for semi-Fredholm operators (see for instance \cite[Theorem 2.42]{taira}) 
we conclude that $D_{\rho,c}^Q$ is semi-Fredholm, i.e., $D_{\rho,c}^Q$ has closed range and a finite dimensional kernel.
Let us now prove that $\Ker (D_{\rho,c}^Q)=\{0\}$. Assume that $u$ is an eigenfunction of $D_{\rho,c}^Q$ with eigenvalue $\lambda \in \R$. Then we have $u^E \in \D(D^Q)$ and may apply \eqref{eq:ibp}.
We obtain 
\begin{align*}
2 \lambda \inprod{u}{\sigma_3 u}_{L^2(\Omega\cap B(c,\rho))} &= \inprod{D^Q u^E}{\sigma_3 u^E}_{L^2(\Omega )} - \inprod{u}{D^Q \sigma_3 u}_{L^2(\Omega )} \\
&= -i \int_{\bdry}    \inprodC{u^E}{\dotprod{\sigma}{n} \sigma_3 u^E} =  \int_{\bdry}    \inprodC{u^E}{\dotprod{\sigma}{t}  u^E},
\end{align*}
where in the last line we used
\begin{equation}\label{eq:sigma.n.t}
-i\dotprod{\sigma}{n}\sigma_3= \dotprod{\sigma}{t}.
\end{equation}
Using the boundary condition $P^Q \gamma u=0$, with $P^Q$ defined in \eqref{eq:Q-dot-BC} for $\eta=\pi/2$, finally gives
\begin{equation}\label{eq:doble.product.eigenvalueQD}
\begin{split}
2 \lambda \inprod{u^E}{\sigma_3 u^E}_{L^2(\Omega)} &=   
\int_{\bdry}    \inprodC{u^E}{ \dotprod{\sigma}{t}  u^E} =
\frac{1}{2}\int_{\bdry}    \inprodC{u^E}{ \seq{\dotprod{\sigma}{t},\id}  u^E} =
   \inprodC{u^E}{ \{\dotprod{\sigma}{t}, A_{\pi/2}^Q \}  u^E} \\
&= \int_{\bdry}    \inprodC{u^E}{  u^E}.
\end{split}
\end{equation}
If $\lambda = 0$, we conclude that $u\in \Ker D_{\rho,c}^Q$ implies that $\gamma u = 0$. In addition, the components of $u$ are (anti-)holomorphic in the interior of $\Omega$, which implies $u \equiv 0$.

We move on to the localized Lorentz scalar operator $D^L_{\rho,c}$. Again, it is symmetric and thanks to \eqref{eq:q-form-D-L}, there exists $C_\bdry>0$ only depending on $\bdry$ such that
to obtain 
\begin{align*}
 \norm{u}_{H^1(\Omega^+ \cap B(c,\rho))\times H^1(\Omega^-\cap B(c,\rho))}^2 
\le&   \norm{D^L u^E}_{L^2(\Omega^+) \times L^2(\Omega^-)}^2 + \norm{u^E}_{L^2(\Omega^+) \times L^2(\Omega^-)}^2+  C_\bdry \norm{\gamma u^E}_{L^2(\bdry)}^2 
\end{align*}
This expression can be bounded in a completely analogous way to obtain
\begin{equation}\label{eq:controlH1_L_loc_op}
\norm{u_+}_{H^1(\Omega^+\cap B(c,\rho))}+\norm{u_-}_{H^1(\Omega^-\cap B(c,\rho))}  \le C \left( \norm{D_{\rho,c}^Q u} + \norm{ u} \right) .
\end{equation}
Reasoning as before, we conclude that $D^L_{c,\rho}$ is closed and semi-Fredholm.

Now if $u$ is an eigenfunction for $D^L_{c,\rho}$ with eigenvalue $\lambda$, we apply the previous identity to $u_+$ and $u_-$ separately to obtain
\begin{align*}
2 \lambda \inprod{u^E}{\sigma_3 u^E}_{L^2(\Omega_+) \times L^2(\Omega_-)} & =  \int_{\bdry}    \inprodC{u_+^E}{\dotprod{\sigma}{t}  u_+} - \inprodC{u_-^E}{\dotprod{\sigma}{t}  u_-}.
\end{align*}
The boundary conditions give
\begin{align*}
\int_{\bdry}   & \inprodC{u_+^E}{\dotprod{\sigma}{t}  u_+^E} - \inprodC{\left(\cosh(\alpha) - \sinh(\alpha) \dotprod{\sigma}{t}  \right)u_+^E}{ \dotprod{\sigma}{t} \left(\cosh(\alpha) - \sinh(\alpha) \dotprod{\sigma}{t}  \right) u_+^E} \\
&= \int_{\bdry}    \inprodC{u_+^E}{\dotprod{\sigma}{t}\left(1 - \cosh^2(\alpha) - \sinh^2(\alpha) \right)  u_+^E} + \inprodC{u_+^E}{ 2 \sinh(\alpha) \cosh(\alpha)  u_+^E} \\
&= 2 \sinh(\alpha) \int_{\bdry}    \inprodC{u_+^E}{\left(- \sinh(\alpha)\dotprod{\sigma}{t} + \cosh(\alpha) \right) u_+^E}.
\end{align*}
Since $\cosh (\alpha) > \abs{\sinh(\alpha)}$, the matrix $- \sinh(\alpha)\dotprod{\sigma}{t} + \cosh(\alpha)$ is positive definite.  As before we deduce that if $\lambda = 0$,  the traces must vanish, which implies again $u_+=u_-=0$.

Finally, the proof of \emph{\ref{item:3-ker}} follows from the same reasoning as the proof of Proposition~\ref{prop:def-D(D^*)} and the fact that $H^2=-\Delta$.
\end{proof}

With these preliminaries, we can prove \Cref{prop:D(D^*)}.
\begin{proof}[Proof of \Cref{prop:D(D^*)}]
Fix $c \in \mathcal{C}$ and $\rho > 0$. We write $\phi_{\rho,c}(x) = \phi((x-c  )/\rho)$.
As we have established in \ref{prop:symmetry}, $\D(D)\subset \D(D^*)$. Let us prove that  $\cN_\rho(c) \subset \D(D^*) $.
We will denote $\mathcal{O}=\Omega$ or $\mathcal O=\Omega^+\times \Omega^-$ depending on the quantum-dot or Lorentz scalar case.
Let $\phi_{\rho,c} u^E \in \cN_\rho(c)$ with $u\in N_\rho(c)$, and $v \in \D(D)$.
Since $ N_{\rho}(c)\subset \D(D_{\rho,c}^*)$ by definition, we find
\begin{align*}
\inprod{D v}{ \phi_{\rho,c} u^E }_{L^2(\mathcal O)} & = 
\inprod{D \phi_{\rho,c}v}{  u^E }_{L^2(\mathcal O)} - \inprod{ v}{ i\left(\dotprod{\sigma}{}\nabla\phi_{\rho,c} \right) u^E }_{L^2(\mathcal O)}\\
&= \inprod{D_{\rho,c}(\phi_{\rho,c}v)}{  u }_{L^2(\mathcal O\cap B(c,\rho)} + \inprod{ i\left(\dotprod{\sigma}{}\nabla\phi_{\rho,c} \right)  v}{ u }_{L^2(\mathcal O\cap B(c,\rho)} \\
&= \inprod{v}{\phi_{\rho,c}D_{\rho,c}^*  u }_{L^2(\mathcal O\cap B(c,\rho)}  + \inprod{ i\left(\dotprod{\sigma}{}\nabla\phi_{\rho,c} \right) v}{  u }_{L^2(\mathcal O\cap B(c,\rho)} \\
&\equiv \inprod{v}{D^*(\phi_{\rho,c}  u^E) }_{L^2(\mathcal O)}.
\end{align*}

Now, we move on to the opposite inclusion.
Fix $w \in \D(D^*)$ and fix a corner $c$ and $\rho>0$. We will show that we can decompose $\phi_{\rho,c} w = \widetilde w^E + \phi_{\rho,c} u^E $, with $\widetilde w \in \D(D_{\rho,c})$ and  $u \in N_\rho(c)$. 
Denoting by $w^R$ the restriction of $w$ to $B(c,\rho)$ we have that $w^R \in \D(D^*_{\rho,c})$ and $D^*_{\rho,c} w^R = (D^* w)^R$.

Now, by \Cref{lem:Fredholm}, $D_{\rho,c}^{-1}:\Ran(D_{\rho,c})\to \D(D_{\rho,c})$ is well defined and bounded and $\Ran(D_{\rho,c})$ is a closed subspace of $L^2$. We decompose $D^*_{\rho,c} w^R$ by projecting on this subspace and its orthogonal:
$D^*_{\rho,c} w^R = D_{\rho,c} \tilde w +  v$,
with $v\in\Ran(D_{\rho,c})^\perp=\Ker(D_{\rho,c}^*)$.
We set $u:=w-\tilde w$ and claim that $u\in\Ker ((D_{\rho,c}^*)^2) $.
Indeed, since both $w^R$ and $\tilde{w}$ belong to $\D(D_{\rho,c}^*)$, we have $u \in \D(D_{\rho,c}^*)$. In addition, we have
\[
H u = H w - H \tilde w = v \in \Ker(D_{\rho,c}^*),
\] 
so $u \in\Ker ((D_{\rho,c}^*)^2)=N_\rho(c)$.
Thus, we have obtained the required decomposition for $\phi_{\rho,c} w$.
If there is more than one corner, we repeat the previous argument with $(1-\phi_{\rho,c} )w$.
Iterating the argument for each corner, we are left with a decomposition
\[
w = \sum_{c \in \mathcal{C}} \left( \widetilde w_c^E + \phi_{\rho,c} u^E\right) + \left( \prod_{c \in \mathcal{C}}(1-\phi_{\rho,c}) \right)u. 
\]
The last term is localized away from all the corners. By the result for smooth domains (see \cite{ourmieres2018strategy}), it is in $H^1$.
\end{proof}

\section{Separation of variables in the wedge} \label{sec:separation}
In this section, we study $N_{\rho}(c)$ for the case that $\Omega \cap B(c,\rho)$ coincides with a truncated wedge with opening angle $\omega$.
We first give some definitions and results. In Subsection 3.1, we obtain a precise description of $N_{\rho}(c)$. In subsection 3.2, we use this description to classify self-adjoint extensions for domains with straight edges close to the corners. At the end of this subsection, we also discuss the behaviour of these extensions under charge conjugation.
Without loss of generality, we can pick coordinates such that $c$ is located at the origin.
In standard polar coordinates $(r, \theta)$ defined by
\[
r:=\sqrt{x_1^2+x_2^2},\quad
\theta:=
\operatorname{sign}(x_2)\arccos\left(\frac{x_1}{r}\right),
\]
the neighborhood of the corner coincides with the wedge $W_\omega$, defined as
\begin{equation}\label{eq:straight_corner}
\Omega \cap B(0, \rho)=W_\omega\cap B(0, \rho),\qquad W_\omega :=\seq{
(r,\theta)\in\C^2: r>0, 0<\theta<\omega}.
\end{equation}
In order to express the Dirac operators in polar coordinates, we define
\[
\bvec{e_r}:=\left(\frac{x_1}{r},\frac{x_2}{r}\right),\quad
\bvec{e_\theta}:=\left(-\frac{x_2}{r},\frac{x_1}{r}\right).
\]
Furthermore, we abbreviate  $\de_r=\bvec{e_r}\cdot{\nabla}$ and $\de_\theta=\bvec{e_\theta}\cdot{\nabla}$,
and obtain
\begin{equation}\label{eq:polar_dirac}
H = -i (\dotprod{\sigma}{e_r} \partial_r + r^{-1}\dotprod{\sigma}{e_\theta}\partial_\theta ) = -i \dotprod{\sigma}{e_r} \left(\partial_r  + i r^{-1}\sigma_3 \partial_\theta \right),
\end{equation}
where in the last equality we use \eqref{eq:sigma.n.t}.

We will need the following functions in order to state our results.
\begin{definition}\label{def:f_k}
Let $\omega\in (0,2\pi)\setminus\seq{\pi}$, and $\alpha\neq 0$.
\begin{itemize}
\item \underline{Quantum-dot}: for all $k\in \Z$, define $f_k^Q:[0,\omega]\to\C^2$ as follows
\[
f_k^Q (\theta) = 
 \frac{1}{\sqrt{2 \omega}\,}\begin{pmatrix}
 e^{i \lambda_k^Q \theta} \\   e^{-i \lambda_k^Q \theta}
\end{pmatrix}, \quad \lambda_k^Q = (2k+1) \frac{\pi}{2 \omega} -1/2.
\]

\item \underline{Lorentz-Scalar}: for all $k\in \Z$ set $f_k^L=(f_{k,+}^L,f_{k,-}^L):[0,\omega]\to\C^4$ with
\[
\begin{split}
f_{k,+}^L(\theta) &:=c_k
 \begin{pmatrix}
\left(\eta_k^L e^{-i\omega/2(\lambda^L_k+1/2)}\right) e^{i \lambda^L_k\theta}\\
\left(i e^{i\omega/2(\lambda^L_k+1/2)}\right) e^{-i \lambda^L_k\theta}\\
\end{pmatrix},\\
f_{k,-}^L(\theta)&:= 
c_k  \begin{pmatrix}
e^{-2i\pi\lambda_k^L}
\left(\eta_k^L\cosh(\alpha) e^{-i\omega/2(\lambda_k^L+1/2)}
-i\sinh(\alpha) e^{i\omega/2(\lambda_k^L+1/2)}\right)
e^{i\lambda_k^L\theta}
\\
e^{-2i\pi\lambda_k^L}
\left(-\eta_k^L\sinh(\alpha) e^{-i\omega/2(\lambda_k^L+1/2)}+i\cosh(\alpha) e^{i\omega/2(\lambda_k^L+1/2)}\right)
e^{-i\lambda_k^L\theta}
\end{pmatrix},
\end{split}
\]
where $\lambda_k^L, k \in \Z$  are the solutions to the transcendental equation 
\[
\abs{\tanh(\alpha)} 
\equiv \frac{2 \abs{\mu}}{1+\mu^2} 
= \frac{\abs{\cos(\pi (\lambda_k^L + 1/2))}}{\abs{\sin((\pi -\omega) (\lambda_k^L+1/2))}} ,
\]
counted in such a way that $\lambda_0^L$ is the unique solution in $(-1/2,0)$ and such that $-\lambda_k^L-1=\lambda^L_{-k-1}$; and
\begin{align*}
\eta^L_k&:=\sign\left(\alpha\sin\left((\pi-\omega)(\lambda_k^L+1/2)\right)\right);
\\
c_k^L&:= e^{i\eta_k^L \pi/4}
\left[2\cosh(\alpha)
	\left(\cosh(\alpha)-\sinh(\alpha)
		\eta_k^L\sin\left(\omega(\lambda_k^L+1/2)\right)\right)
		\right]^{-1/2}.
\end{align*} 
\end{itemize}
Finally let 
\begin{equation}\label{eq:def:u_k}
u_k(r,\theta):=
\phi(r/\rho) r^{\lambda_k} f_k(\theta).
\end{equation}
being $\phi$ the cut-off  function defined in \eqref{eq:def.cutoff}.
\end{definition}

With all definitions in place, we can give a precise description of $\cN_\rho(0)$.
\begin{theorem} \label{prop:quotient}
Let $\omega\in (0,2\pi)\setminus\seq{\pi}$, $\Omega$ as in \eqref{eq:straight_corner}. Let $\cN_\rho(0)$ be defined as in \eqref{eq:def.N_rho} and $u_k$ be defined as in \eqref{eq:def:u_k}. 
Then

\[
\cN_\rho(0) \diagup \left(\cN_\rho(0) \cap H^1\right) = \Span \{ u_k : \lambda_k \in (-1,0] \}.
\]
\end{theorem}
For $\omega < \pi$ (a convex corner), there are none of the $\lambda^Q_k$'s in $(-1, 0]$, while for $\omega > \pi$, we have only $ \lambda^Q_0$ and $\lambda^Q_{-1}$ in $(-1,0]$.
In the Lorentz-Scalar case,  $ \lambda^L_0$ and $\lambda^L_{-1}$ lie in $(-1,0]$, regardless of the value of $\omega$.
Thanks to this and combining \Cref{prop:D(D^*)} and \Cref{prop:quotient} we directly have the following results:

\begin{proposition}\label{prop:quotient-adj} 
Let $\omega\in (0,2\pi)\setminus\seq{\pi}$, let $\Omega $ be a piecewise $C^2$ domain with a single, straight corner of opening $\omega$, that is $\Omega$ verifies \eqref{eq:straight_corner}. Let $D^Q$ and $D^L$ be defined respectively as in \eqref{eq:def-dom-DQ-R} and \eqref{eq:def-dom-DL-R} and let $u_k$ be defined  as in \Cref{def:f_k}. Then
\begin{itemize}
\item \underline{Quantum-dot}:
 \begin{enumerate}[label=(\roman*)]
\item for $0 < \omega < \pi$:
\[
\D((D^Q)^*) = \D(D^Q);
\]
\item for $\pi < \omega <2 \pi$:
\[
\D((D^Q)^*) = \D(D^Q) + \operatorname{span} (u_0^Q, u_{-1}^Q),
\]

\end{enumerate}
\item[]

\item \underline{Lorentz-Scalar}:
For $\omega \neq \pi$: 
\[
\D((D^L)^*) = \D(D^L) + \operatorname{span} (u_0^L, u_{-1}^L).
\]
\end{itemize}
\end{proposition}

\subsection{Proof of \texorpdfstring{\Cref{prop:quotient}}{Theorem 3.2}}
In order to prove \Cref{prop:quotient}, we need the following lemma, whose proof is left for \Cref{sec:appendix-angular}.
\begin{lemma}\label{lem:basis-eigen}
Let $\seq{f_k(\theta)}_{k\in\Z}$ and $\seq{\lambda_k}_{k\in\Z}$ be defined as in \Cref{def:f_k} and
define the angular operators $D^Q_{ang}$ and $D^L_{ang}$ as the action of $-i\sigma_3\partial_\theta$ on
\begin{align*}
\D(D^Q_{ang})&:=\seq{
f\in H^1([0, \omega], \C^2):
f_2(0) = f_1(0), \ f_2(\omega) = -  e^{i \omega} f_1(\omega)},
\\
\D(D^L_{ang})&:=
\seq*{
\begin{array}{l}
(f^+,f^-)\in H^1([0, \omega], \C^2)\times H^1([\omega,2\pi], \C^2):\\ \\
f_-(2\pi)=\left(
\begin{smallmatrix}  
\cosh(\alpha) & - \sinh(\alpha) \\ -\sinh(\alpha) & \cosh (\alpha)
\end{smallmatrix} 
\right)f_+(0),\quad
f_-(\omega) = 
\left(
\begin{smallmatrix}
  \cosh(\alpha) & e^{-i \omega} \sinh(\alpha) \\ e^{i \omega} \sinh(\alpha) & \cosh (\alpha)\end{smallmatrix}\right)
   f_+(\omega)
\end{array}};
\end{align*}
respectively, in the Hilbert spaces
\begin{align*}
\H^Q&:=\left(L^2([0,\omega],\C^2);\langle\cdot,\cdot\rangle_{L^2}\right),\qquad
\H^L:=\left(L^2([0,\omega],\C^2)\times L^2([\omega,2\pi],\C^2); \langle\cdot,\cdot\rangle_{L^2}\right).
\end{align*}
Then $D_{ang}$ is a self-ajdoint operator with an orthonormal basis of eigenfunctions $\seq{f_k}_{k\in\Z}$ and eigenvalues $\lambda_k$.
Moreover, for any $k\in\Z$
\begin{equation}\label{eq:sigma.er.f_k}
\dotprod{\sigma}{e_r}f_k=f_{-k-1}.
\end{equation}
\end{lemma}


We can now prove \Cref{prop:quotient}
\begin{proof}[Proof of \Cref{prop:quotient}]
Take $u \in N_\rho(0)$. It has a decomposition in angular eigenfunctions
\[
u(r,\theta) = \sum_{k\in \Z} h_k(r) f_k(\theta) \quad \text{ with } (h_k)_{k\in \Z} \in \ell^2(\Z, \H ).
\]
Since $\Delta u = 0$, which reduces to
\[
h_k''(r) + r^{-1} h_k'(r) -r^{-2}\lambda_k^2 h(r) = 0.
\]
The solution of this equation is $h_k(r) = a_k r^{\lambda_k} + b_k r^{-\lambda_k}$ with some coefficients $a_k, b_k$ in $\C$.
Since $u \in L^2$, we have $a_k = 0$ if $\lambda_k \le - 1$ and $b_k = 0 $ if $\lambda_k \ge 1$.
Therefore, we obtain
\[
u(r,\theta) = \sum_{\lambda_k >- 1} a_k r^{\lambda_k} f_k(\theta) +  \sum_{\lambda_k <  1} b_k r^{-\lambda_k} f_k(\theta).
\]
We also now that $H u = D^*_{\rho,0} u$ is in $L^2$.
By construction, 
\[
H h(r) f_k(\theta) = -i \dotprod{\sigma}{e_r} \left(\partial_r  - D_{\rm ang} \right) h(r) f_k(\theta) = -i(h'(r) - \lambda_k r^{-1} h (r)) f_{-k}(\theta),
\]
so we obtain
\begin{equation} \label{eq:Hu}
H u (r,\theta) = \sum_{\lambda_k <  1} (-2i \lambda_k )b_k r^{-\lambda_k-1} f_{-k}(\theta).
\end{equation}
In order to be square integrable close to the origin, we need $b_k = 0$ for all values of $k$ with $\lambda_k > 0 $.
Therefore, we have obtained a decomposition
\[
u(r,\theta) = \sum_{\lambda_k \in (-1, 0)} a_k r^{\lambda_k} f_k(\theta) + \sum_{\lambda_k \ge 0} a_k r^{\lambda_k} f_k(\theta)+  \sum_{\lambda_k \le  0} b_k r^{-\lambda_k} f_k(\theta).
\]
Individual terms in each of the last two series are in $H^1$, but we still need to show that the same holds true for the sum.
We write
\[
v = \sum_{\lambda_k \ge 0} a_k r^{\lambda_k} f_k(\theta), \qquad
w =  \sum_{\lambda_k \le  0} b_k r^{-\lambda_k} f_k(\theta)
\]
and treat each of them separately.
For $w$, we use the orthonormality of the angular functions to write (with the understanding that both sides may equal $+\infty$)
\begin{align*}
\norm{\nabla w}_{L^2}^2 
&= \norm{\partial_r w}_{L^2}^2 + \norm{r^{-1}D_{\rm ang} w}_{L^2}^2 \\
&= \sum_{\lambda_k \le 0} \abs{b_k}^2 2\lambda_k^2 \norm{r^{-\lambda_k-1} f_k(\theta)}_{L^2}^2 \\
&= 1/2 \sum_{\lambda_k \le 0} \norm{ Hr^{-\lambda_k-1} f_k(\theta)}_{L^2}^2 = 1/2 \norm{Hu}^2_{L^2}  ,
\end{align*}
where the last line follows from \eqref{eq:Hu}.
Since $H u$ is square integrable, this shows that $w \in H^1$.
Now for $v$, we use the fact that $u$ is in $H^1$ when localized away from the corner and from $\partial B(0, \rho)$, by the result for smooth domains.
Restricting to the Lorentz scalar case for simplicity of notation, we have that 
\begin{align*}
 \norm{\nabla v}^2_{L^2(\Omega \cap B(0,2\rho/3) \setminus B(0, \rho/3))} 
&=  \sum_{\lambda_k > 0} \abs{a_k}^2 2\lambda_k^2 \norm{r^{\lambda_k-1} f_k^Q(\theta)}_{L^2(\Omega \cap B(0,2\rho/3) \setminus B(0, \rho/3))}^2 \\
&= \sum_{\lambda_k > 0} \abs{a_k}^2 2\lambda_k^2 (2 \lambda_k)^{-1} \left((2\rho/3)^{2 \lambda_k} - (\rho/3)^{2 \lambda_k} \right) \\
& \ge 1/C \sum_{\lambda_k > 0} \abs{a_k}^2 2\lambda_k^2 (2 \lambda_k)^{-1} (2\rho/3)^{2 \lambda_k} \\
&= 1/C \norm{\nabla v}^2_{L^2(\Omega \cap B(0,2\rho/3) )}. 
\end{align*}
Here, $1/C \equiv 1- (1/2)^\lambda$ with $\lambda$ the smallest positive $\lambda_k$.
This is sufficient to conclude, since $\supp (\phi_{\rho,0}) \subset B(0, 2\rho/3)$.
\end{proof}

\subsection{Characterization of self-adjoint extensions of \texorpdfstring{$D^Q$}{DQ} and \texorpdfstring{$D^L$}{DL}}

We can now describe all self-adjoint extensions of $D^Q$ and $D^L$ for domains with straight edges in a neighborhood of each corner.
For simplicity, we state the theorem for domains with a single corner, but the generalization is straightforward.

\begin{theorem}\label{thm:wedge-full}
Let $\omega\in (0,2\pi)\setminus\seq{\pi}$, $\Omega$ as in \eqref{eq:straight_corner}. Let $D^Q$ and $D^L$ be defined respectively as in \eqref{eq:def-dom-DQ-R} and \eqref{eq:def-dom-DL-R} and let $u_k$ be defined  as in \Cref{def:f_k}. Then

\begin{itemize}
\item \underline{Quantum-dot}:
 \begin{enumerate}[label=(\roman*)]
\item for $0 < \omega < \pi$, $D^Q$ is self-adjoint

\item for $\pi < \omega <2 \pi$, $D^Q$ admits infinite self-adjoint extensions, and they all belong to the one-parameter family $\seq{D_\tau^Q}_{\tau\in[0,\pi)}$ with domains
\begin{equation}\label{eq:def_D(D^Q_tau)}
\D(D^Q_\tau) = \D(D^Q) + \operatorname{span} (\cos(\tau) u_0^Q +i \sin(\tau) u_{-1}^Q).
\end{equation}
\end{enumerate}

\item[]

\item \underline{Lorentz-Scalar}:
$D^L$ has infinite self-adjoint extensions, and they all belong to the one-parameter family $\seq{D_\tau^L}_{\tau\in[0,\pi)}$ with domains
\begin{equation}\label{eq:def_D(D^L_tau)}
\D(D^L_\tau) = \D(D^L) + \operatorname{span} (\cos(\tau) u_0^L +i \sin(\tau) u_{-1}^L).
\end{equation}

\end{itemize}
\end{theorem}

\begin{proof}
If $0 < \omega < \pi$, then $D^Q$ is self-adjoint thanks to \Cref{prop:quotient-adj}

Since the approach is the same, now we analyse at the same time 
 the quantum-dot operator in the case that $\pi<\omega<2\pi$, and the Lorentz-scalar operator. 
 
Let $u\in \D(D^*)$. Then, by \Cref{prop:quotient-adj}, $u=\tilde u+c_0 u_0+c_{-1} u_{-1}$, with $\tilde u\in \D(D)$ and $c_0 ,c_{-1}\in \C$. 
Since $u_0, u_{-1}\in D(D^*)$ and due to the symmetry of $D$ we have that
\begin{multline*}
\inprod{D^*u}{u}_{L^2}-\inprod{u}{D^*u}_{L^2}\\
=
\inprod{c_0D^*u_0+c_{-1} D^*u_{-1}}{c_0 u_0+c_{-1} u_{-1}}_{L^2}-\inprod{c_0 u_0+c_{-1}  u_{-1}}{c_0 D^*u_0+c_{-1} D^*u_{-1}}_{L^2}.
\end{multline*}
Combining \eqref{eq:polar_dirac}, \eqref{eq:sigma.er.f_k} we have that
\begin{equation}\label{eq:D^*u}
D^*u_k(r,\theta)=-\frac{i}{\rho} \phi'(r/\rho) r^{\lambda_k} f_{-k-1}(\theta),\quad\text{for}\ k=0,1.
\end{equation}
Due to the orthonormality of  $f_0$ and $f_1$, and since $\phi(0)=1$, one has
\begin{equation}\label{eq:commut-D^*u_k,u_l}
\inprod{D^*u_k}{u_l}=
\begin{cases}
0&\text{if}\ k=l\\
i/2&\text{if}\ k\neq l
\end{cases}
\qquad
\text{for} \ k,l=0,-1.
\end{equation}
Thanks to this, we can conclude that
\begin{equation}\label{eq:Symmetry}
\inprod{D^*u}{u}_{L^2}-\inprod{u}{D^*u}_{L^2}=
2i\Re(c_0\overline{c_{-1}}).
\end{equation}
Let now $\tilde D$  be a non-trivial symmetric extension of $D$, that is $D\subsetneq \tilde D$.
From \eqref{eq:Symmetry} we have if $u\in \D(\tilde{D})$, then $\Re(c_0\overline{c_{-1}})=0$. Following for instance \cite[Lemma 3.2]{coulombbiagio}, we conclude that  there exists $\tau\in[0,\pi)$ such that
$i c_0\sin \tau+ c_{-1}\cos\tau=0$, that is equivalent to say that $c_0= c\cos\tau$ and $c_{-1}=i c\sin \tau$, for an appropriate $c\in \C$.
This means that the operator $D_\tau$ defined in \eqref{eq:def_D(D^Q_tau)} is symmetric, and that if $\tilde D$ is a non-trivial symmetric extension of $D$, then $\tilde D=D_\tau$ for a certain $\tau\in[0,\pi)$.

Let us prove that $D_\tau$ is self-adjoint.
By construction 
\[
D \subsetneq D_\tau \subset D_\tau^* \subset D^*.
\]
Let $v=\tilde{v}+c_0 u_0+c_{-1} u_{-1}\in \D(D^*_\tau)$ and take $u=(\cos\tau u_0+i\sin\tau u_{-1}) \in \D(D_\tau)$, with $\tilde{v}\in\D(D)$ and $c_0,c_{-1}\in\C$. Reasoning as before and thanks to \eqref{eq:commut-D^*u_k,u_l} we have that
\[
0=\inprod{D^*_\tau v}{u}_{L^2}-\inprod{v}{D_\tau u}_{L^2}
=i/2\left(-i \overline c_0\sin\tau+ \overline c_{-1}\cos\tau\right),
\]
that directly implies that $c_0=c \cos\tau$ and $c_{-1}=i c\sin\tau$ for an appropriate $c\in\C$. Then $v \in\D(D_\tau)$, and so $D_\tau$ is self-adjoint.
\end{proof}

For the cases where there are infinitely many self-adjoint extensions, we always have $\lambda_0 \in (-1/2, 0)$ and $\lambda_{-1} \in (-1, -1/2)$.
Therefore, $u_0$ is in $H^{1/2}$, while $u_{-1}$ is not (see for instance \cite[Theorem 1.4.5.3]{grisvard} for a proof of this).
Thus, the restriction of $D^*$ to $H^{1/2}$ coincides with $D_{\tau = 0}$, and we have proven Theorem~\ref{thm:main} for the case of corners with straight edges.
With our notation, $D_0$ is the self-adjoint extension of $D$ with the most \emph{regular} domain.
An other criterion to select an extension is invariance under charge conjugation, as proposed in \cite{le2018self}.

In the model that we consider here, the anti-unitary operator of charge conjugation is given by 
\begin{equation}\label{eq:def.ChargeConj}
C u:=\sigma_1 \overline{u}.
\end{equation}
When dealing with 4-spinors, charge conjugation is related to the particle-antiparticle interpretation of the Dirac field, see \cite[Section 1.4.6]{thaller}.
In our model, it is just a composition of time reversal (complex conjugation) and parity transformation (swapping spinor components).

The charge conjugation operator anti-commutes with the free Dirac operator in $\R^2$ and also with its perturbation by mass terms of the form $m(x) \sigma_3$, where $m$ can be any real function.
Since the quantum dot operator for $B=1$ and the Lorentz scalar delta-shell operator are limits of operators of this type, it is natural to expect that they are invariant as well.
Indeed, a short computation suffices to verify that $C \D(D) = \D(D)$.

In Proposition \ref{propr:Charge u_k}, we show that, with our choice of phase factor, $C u_k = u_k$. 
Thus,
\[
C(\D(D_\tau))=
\D(\tau)+\operatorname{span}(\cos(\tau)u_0-i\sin(\tau)u_{-1})=\D(D_{-\tau}).
\]
So we conclude that $CD_\tau=-D_{-\tau}C$. This means that $D_0$ and $D_{\pi/2}$ are the only extensions that anti-commute with charge conjugation.

\section{Curvilinear polygons} \label{sec:curvilinear}

In this section, we deduce \Cref{thm:main} from \Cref{thm:wedge-full}.
We first check that the operators with domains in $H^{1/2}$ are symmetric.
In order to simplify the notation further, we assume that $\Omega$ has a single corner centred at the origin. The case of several corners is again just a matter of extra notation.

\begin{lemma} \label{lem:extension-symmetry}
The operators $D_0^L$ and $D_0^Q$, as defined in \Cref{thm:main} are symmetric.
\end{lemma}
\begin{proof}
The proof is identical for the Lorentz scalar and quantum-dot case. We give it here for the latter case, since the notation is more concise. 
Fix $u, v \in \D(D_0^Q)$. By the result for smooth domains, $u$ and $v$ are in $H^1(\Omega \setminus B(0,r))$ for all $r >0$. 
We apply \eqref{eq:ibp} from Lemma~\ref{lem:ibp} to conclude that
\[
 \inprod{u}{Hv}_{L^2(\Omega \setminus B(0,r))} - \inprod{Hu}{v}_{L^2(\Omega \setminus B(0,r))}  = -i\int_{\bdry \setminus B(0,r)} \inprodC{u}{\bvec \sigma \cdot \bvec n_\Omega v} 
 -i\int_{ \Omega  \cap \partial B(0,r)} \inprodC{u}{\bvec \sigma \cdot \bvec e_r v}. 
\]
The first term vanishes because of the boundary conditions, that hold in the classical sense away from the corner.
In order to estimate the second term, we average the identity over $r \in [s, 2s]$. This gives
\[\begin{split}
 \frac{1}{s} \int_s^{2s} \abs*{\int_{ \Omega  \cap \partial B(0,r)} \inprodC{u}{\bvec \sigma \cdot \bvec e_r v}} \di r
& \le  \frac{1}{s} \int_{\Omega \cap B(0,2s) \setminus B(0,s)} \abs{u}\abs{v} \\
& \le \frac{1}{s} \norm{u}_{L^4(\Omega\cap B(0,2s))}  \norm{v}_{L^4(\Omega\cap B(0,2s))} \abs{\Omega\cap B(0,2s))}^{1/2} \\ 
& \le C \norm{u}_{L^4(\Omega\cap B(0,2s))}  \norm{v}_{L^4(\Omega\cap B(0,2s))}.
\end{split}\]
The final bound tends to zero as $s \to 0$, since $u, v  \in L^4(\Omega)$ by the Sobolev embedding $H^{1/2}(\Omega) \subset L^4(\Omega) $. 
\end{proof}

This reduces the problem of self-adjointness to the issue of showing that the domain of the adjoint operator stays in $H^{1/2}$. We know as well that the domain of the adjoint is included in the maximal domain, so away from the corners, elements in the domain of the adjoint are even $H^1$.
Close to the corners, we have to transform coordinates to straighten the boundary. In general, this transformation set up a unitary equivalence between the Dirac operator on the curvilinear wedge and the Dirac operator plus a perturbation on the straight wedge.
The unbounded part of this perturbation consists of derivatives of the first order, multiplied by a function that measures the difference between the Jacobian matrix of the coordinate transformation and the identity matrix.
In the case of smooth boundaries, this perturbation is irrelevant by the elliptic regularity of the Dirac operator on the half-space.

Here, elliptic regularity does no longer hold, so we need a to work a little bit more. 

For the quantum-dot case, we can avoid issues by using bounded conformal transformation to send the interior of the domain to a subset of the wedge. This conformal transformation maps the maximal domain to the maximal domain on the wedge, where the classification from Theorem~\ref{thm:wedge-full} remains valid. This allows for a classification of self-adjoint extensions for the quantum-dot operator as well.
For the sake of brevity, we have stated Theorem~\ref{thm:main} for the extension with domain in $H^{1/2}$, and give a single proof that applies to both the Quantum-Dot and the Lorentz-scalar model.

The case of the Lorentz scalar operator is more delicate, because it is not, in general, possible to find a conformal transformation that maps both the interior and the exterior of the curvilinear domain to the interior and exterior of the wedge.
On the other hand, it is always possible to find a $C^2$ coordinate transformation that achieves this, 
but in this case, we have to treat the perturbation terms carefully.
 We choose a coordinate transformation with the perturbation of the Jacobian matrix of order $r$, with $r$ the distance to the corner. Combined with the $H^{1/2}$ regularity in the whole domain, this gives us precisely what is needed to conclude. The perturbation terms are finite and symmetric on the image of the original domain, which allows to conclude that the image of the original domain is included in $\D (D^L_0)$ on the wedge. 
By using the decomposition of spinors in this domain in a $H^1$ part and a multiple of $u_0^L$, we conclude that the perturbation terms are relatively bounded with respect to the full operator, with a relative bound that can be made smaller by taking a smaller neighbourhood of the corner.  
Note that this strategy does not give a classification of self-adjoint extensions, it only proves the existence of a single extension with the domain in $H^{1/2}$.

We write $D_0^{L,\Omega}$ and $D_0^{L,W_\omega}$ to distinguish the operators $D^L_0$ acting on $\Omega$ and $W_\omega$ respectively.
A first technical step is to construct a coordinate transformation that maps the curved boundary inside this boundary to a straight boundary. An explicit example is given in Appendix~\ref{sec:coordinate-transf}
Having this transformation at hand, we also have to transform spinors so that the transplanted functions satisfy the boundary conditions on the new domain. 
This is achieved by means of point-wise multiplication by a matrix that, at the boundary points, equals $e^{i\sigma_3 \gamma/2}$, where $\gamma$ is the angle measuring the rotation to pass from the tangent vector to the curved boundary to the tangent vector at the boundary of the wedge.
We denote this transformation by $U$. The map $U$ can be chosen to be unitary.
Again, details of this transformation can be found in \Cref{sec:coordinate-transf}. 

The result of this rather technical construction is to set up a unitary equivalence between $D^{L, \Omega}_0$ (after restriction to a neighbourhood of the corner), and an operator in the wedge, that decomposes as
\begin{equation}\label{eq:def_perturbation}
U D^{L, \Omega}_0 U^* = H + \sum_{j=1,2} L_{j}(x) \partial_j + M (x),
\end{equation}
with $L_j$ and $M$ defined in \eqref{eq:def-L-M}.
The matrices $L_j$ depend on the difference between the Jacobian matrix of the coordinate transformation and the identity. The matrix $M$ is a multiplication operator containing first and second derivatives of the functions giving the transformation.  
By the $C^2$ regularity of the boundary, $M$ is bounded, and the the transformation can be chosen to tend linearly to the identity when approaching the origin.
A priori, the expression $H + \sum_{j} L_j(x) \partial_j $ has to be taken in distribution sense, where only the sum of both is well-defined on $U \D(D^{L,\Omega}_0)$.

What we use in the following, is that
\begin{equation}\label{eq:bound_perturbation}
 \norm{L_j(x) }_{\C^2 \to \C^2} \leq C|x|, \quad C>0.
\end{equation}
We first check that $U D^{L, \Omega}_0 U^*$, given by the differential expression \eqref{eq:def_perturbation}, is well-defined on $\D(D^{L,W_\omega}_0)$

 \begin{lemma} \label{lem:pert-bound}

 Let $\Omega $ be a corner domain and assume the origin is at a corner. Assume that $u \in \D(D_0^{L,\Omega})$ with support in $B(0,R)$, where $R>0$ is sufficiently small such that the origin is the only corner in $\supp u$.
 Then
 \[
\norm{\abs{x} \nabla u(x)}_{L^2} \le \infty. 
 \]
 \end{lemma}
 \begin{proof}
 First, we note that
 \[
 \norm{\abs{x} \nabla u(x)}_{L^2} \le  \norm{ \nabla\abs{x} u(x)}_{L^2} + \norm{u}_{L^2}^2, 
 \]
 where the second term is clearly finite.
 Since $u$ is $H^1$ away from the origin, we may use \eqref{eq:q-form-general} and \eqref{eq:q-form-D-L} from Lemma~\ref{lem:ibp} to write, for any $r>0$,
 \[\begin{split}
\int_{\R^2\setminus B(0,r)} &\abs{\nabla \abs{x} u(x)}^2 \di x \\
&= \int_{\R^2\setminus B(0,r)} \abs{D^L (\abs{x} u(x))}^2 \di x - \int_{\bdry \setminus B(0,r)} \kappa \abs{s}^2 \inprodC{u_+}{\sigma_3 u_+}(s)\di s \\
&+ \int_{\partial B(0,r) \cap \Omega}  r^2 \inprodC{u_+}{i \sigma_3\partial_{\theta} u_+}(r,\theta)  \di \theta \\
 &+ \int_{\partial B(0,r) \setminus \Omega}  r^2 \inprodC{u_-}{i \sigma_3\partial_{\theta} u_-}(r,\theta)\di \theta .
 \end{split}\]
Since $D^L (\abs{x} u(x)) = \abs{x} D^L u + \abs{x}^{-1}\dotprod{\sigma}{ x } u$, the first term is bounded independently of $r$.
The second term is bounded as well, by using the representation of the boundary traces given in the proof of Lemma~\ref{lem:trace-ext}. Indeed, from \eqref{eq:trace-ext-def}, $T$ is bounded from $H^s(\Omega)$ to $H^{s-1/2}(\Omega)$ for all $s\le 1$, and thus, $u_+ \in H^{1/2}(\Omega,\C^2)$ has boundary traces in $L^2 (\bdry,\C^2)$. 

In order to estimate the contribution from the boundary of $B(0,r)$, we average over $r \in (0,t)$ and write
\[\begin{split}
\frac{1}{t}\int_0^t \int_{\partial B(0,r) \cap \Omega}  r^2 \abs{u_+} \abs{\partial_\theta u_+} (r,\theta) \di \theta \di r
&\le \int_{B(0,t)\cap \Omega}\abs{u_+} \abs{\partial_\theta u_+} (r,\theta) r \di \theta \di r \\
&\le \norm{u_+}_{H^{1/2}}^2.
\end{split}\] 
The same argument works for $u_-$.
Putting everything together, we have shown that, for all $t >0$,
\[
\frac{1}{t} \int_0^t\norm{ \nabla\abs{x} u(x)}_{L^2(\R^2 \setminus B(0,r)}^2 \di r \le 
C(\norm{D^L_0 \abs{x} u(x)}_{L^2} + \norm{u}_{H^{1/2}}).
\]
Since $\norm{ \nabla\abs{x} u(x)}_{L^2(\R^2 \setminus B(0,r)}^2$ increases as $r$ decreases, this shows that the limit at zero is finite. 
 \end{proof}

The previous lemma shows that $U$ maps $\D(D^{L, \Omega}_0)$ unitarily into $\D(D^{L, W_{\omega}}_0)$. 
We can now use Theorem~\ref{thm:wedge-full} to conclude that $u \in \D(D^{L, \Omega}_0)$ decomposes as
\[
u = v + c_0 U^{-1} u_0^L
\]
for some $v \in H^1$ and $c_0 \in \C$. 
This decomposition also allows to show that the second term in \eqref{eq:def_perturbation} is relatively bounded with respect to the first one.

\begin{lemma} \label{lem:pert-rel-bound}
Let $u \in \D(D^{L, W_{\omega}}_0)$.
Then we have
\[
\norm*{ \sum_{j} L_j(x) \partial_j u(x)}_{L^2} \le C\left(\sup_{x \in \supp u } |x|\right)  \norm{D^{L, W_{\omega}}_0 u}_{L^2} +2C   \,\norm{u}_{L^2},\quad\text{with}\ C>0.
\]
\end{lemma} 
\begin{proof} 
By \Cref{prop:quotient}, any $u \in \D(D^{L, W_{\omega}}_0)$ has a decomposition $u= v + c_1 u_0^L$ with $v \in H^1$. The key point is that the entries of $\abs{x} u_0^L(x)$ behaves as $r^{\lambda_0^L +1/2} e^{ \pm i (\lambda^L_0-1/2) \theta}$ and therefore, $|x|u(x)$ is in $H^1$. In addition, $u$ satisfies boundary conditions, so $ |x| u \in \D(D^L)$.
Now, we use \eqref{eq:bound_perturbation} and \eqref{eq:q-form-D-L} with $\kappa = 0$ to bound
\[\begin{split}
\norm*{ \sum_{j= 1,2} L_{j} \partial_j u}_{L^2} 
\le  C \norm{|x|\nabla u}_{L^2} 
& \le C\norm{\nabla (|x| u)}_{L^2} + C \norm{u}_{L^2}       \\
&\le  C\norm{D^L |x| u}_{L^2} +  C \norm{u}_{L^2}      \\
& \le C \norm{|x| D^L_0 u}_{L^2} +  2C \norm{u}_{L^2}.   \qedhere
\end{split}\]
\end{proof}

\begin{proof}[Proof of Theorem ~\ref{thm:main}]
We start by localizing the Dirac operator with an IMS-type formula.
Fix a cutoff $\phi$ as in $\eqref{eq:def.cutoff}$. For $\rho >0$ small enough, we write $\phi_\rho(x) = \phi(x)/\rho$. 
Then 
\begin{equation}\label{eq:magic-H}
H=\phi_\rho H\phi_\rho +\sqrt{1-\phi_\rho^2}H\sqrt{1-\phi^2}.
\end{equation}
After replacing $H$ by $D_0^{L,\Omega}$ in \eqref{eq:magic-H}, the second addend describes a self-adjoint operator, since the corner does not belong to the support of $1-\phi_\rho^2$.

We now focus on the first addend. Since we are considering only functions that are localized close to the corner, we can assume that $\Omega=W_\omega$ outside a sufficiently large neighbourhood of the origin.
Let $U$ be the unitary transformation defined in \Cref{sec:coordinate-transf};
by \eqref{eq:def_perturbation},
\eqref{eq:bound_perturbation} and  Lemma~\ref{lem:pert-bound} we have
$U \D(D_0^{L, \Omega}) = \D(D_0^{L, W_{\omega}})$. 
In particular, from \eqref{eq:bound_perturbation} we have that
\begin{equation}\label{eq:almost-done}
U \phi_\rho D_0^{L, \Omega}\phi U^* =
\widetilde\phi_\rho D_0^{L, W_{\omega}}\widetilde\phi + \widetilde\phi_\rho \sum_{j=1,2} L_j \partial_l \widetilde\phi_\rho+ \widetilde\phi_\rho M\widetilde\phi_\rho,
\end{equation}
with $\widetilde\phi_\rho=\phi_\rho\circ S^{-1}$ and $S$ the coordinate transformation defined in \Cref{sec:coordinate-transf}.
The last two terms of the right-hand-side of \eqref{eq:almost-done} are symmetric on $\D(D_0^{L, W_{\omega_j}})$  since both $D_0^{L, \Omega}$ and $D_0^{L, W_{\omega}}$ are symmetric.
In addition, 
\[
\begin{split}
||\widetilde \phi_\rho L_j\partial_j\widetilde\phi_\rho u||_{L^2}
&\leq
||L_j\partial_j \widetilde \phi_\rho^2 u||+ C||\nabla\phi_\rho||_{L^\infty}||u||_{L^2}\\
&\leq C \left(\sup_{x\in\supp\phi} |x|\right)
||\widetilde\phi D^{L,W_\omega}_0 \widetilde\phi_\rho u||_{L^2}+2C||\nabla\phi_\rho||_{L^\infty}||u||_{L^2},
\end{split}
\]
where in the last line, we used \Cref{lem:pert-rel-bound}. So we have that $
\sum_{j} \widetilde \phi_\rho L_j\partial_j\widetilde\phi_\rho+\widetilde\phi_\rho M\widetilde\phi_\rho$ is relatively bounded with respect to $\widetilde\phi_\rho D^{L,W_\omega}_0 \widetilde\phi_\rho$. Choosing $\rho$ sufficiently small, the relative bound can be made smaller than $1$ and so, 
by the Kato-Rellich theorem, see \cite[Theorem 4.3]{kato2013perturbation}                  
for instance,
we conclude that $ \phi_\rho D_0^{L, \Omega} \phi_\rho$ is unitarily equivalent to a self-adjoint operator .
\end{proof}

\appendix

\section{Some technical identities}\label{sec:appendix}
In this Appendix, we prove some technical results from Section~\ref{sec:general}. 
We start with \Cref{lem:ibp}:
\begin{proof}[Proof of Lemma~\ref{lem:ibp}]
The identity \eqref{eq:ibp} follows from the divergence theorem.
Let us prove \eqref{eq:q-form-general}. For $u,v \in H^2(\mathcal{O}, \C^2)$, identity \eqref{eq:q-form-general} follows from writing
\begin{align*}
\inprodC{Hu}{Hv} &= \sum_{j,k} \inprodC{\sigma_j \partial_j u}{ \sigma_k\partial_k v} \\
& =  
\sum_{j} \inprodC{\sigma_j \partial_j u}{ \sigma_j\partial_j v}+ 
\sum_{j\neq k} \inprodC{\sigma_j\partial_j u}{\sigma_k\partial_k v}\\
&=:I+II.
\end{align*}
Then $I=\inprod{\nabla u}{\nabla v}_{L^2(\mathcal{O})}$, because the matrices $\sigma_j$ are symmetric and $\sigma_j^2=\id$.
For the second term, since $\sigma_1\sigma_2=-\sigma_2\sigma_1=i\sigma_3$, by the divergence theorem we have that
\begin{align*}
II&=
\sum_{j\neq k} \int_\mathcal{O} \inprodC{\partial_j u}{\sigma_j \sigma_k\partial_k v}
= 
\int_\mathcal{O} \inprodC{\partial_1 u}{i\sigma_3 \partial_2 v}-
\int_\mathcal{O} \inprodC{\partial_2 u}{i\sigma_3 \partial_1 v}\\
&=
\int_{\partial \mathcal{O}}
\inprodC{u}{i \sigma_3 (\bvec {n_1} \partial_2-\bvec {n_2}\partial_1)  v}.
\end{align*}
Finally, in dimension two:
\[
\left( \bvec{n}_1 \partial_2 - \bvec{n}_2 \partial_1 \right) = \bvec t \cdot \nabla,
\]
that gives the required identity for $u,v \in H^2(\mathcal{O}, \C^2)$.
By density, it extends to $u, v \in H^1(\mathcal{O}, \C^2)$, upon interpreting the boundary term as the pairing between $u \in H^{1/2}(\partial \mathcal{O}, \C^2)$ 
and $\sigma_3 \partial_{\bvec t} v \in H^{-1/2}(\partial \mathcal{O}, \C^2)$.

Identity \eqref{eq:q-form-D-Q}
 is just \eqref{eq:q-form-general} restricted to functions satisfying boundary conditions, which allows to rewrithe the boundary term in a convenient form. Recall that $\kappa$ is defined as a piecewise continuous function on $\bdry$.
We firstly treat the case $B= 1$ (infinite mass).
For functions $f, g \in C^1(\bdry, \C^2)$ that satisfy infinite mass boundary conditions
\[
f  = \dotprod{\sigma}{t} f, \quad g  = \dotprod{\sigma}{t} g, 
\]
we have, point-wise away from the corners,
\begin{align*}
\inprodC{f}{i\sigma_3\partial_{\bvec t}  g}
& = \frac{1}{2} \inprodC{\dotprod{\sigma}{t}f}{i\sigma_3\partial_{\bvec t} g} +\frac{1}{2} \inprodC{f}{i\sigma_3\partial_{\bvec t} (\dotprod{\sigma}{t} g)} \\
&= \frac{1}{2} \inprodC{\dotprod{\sigma}{t} f}{i \partial_{\bvec t}  \sigma_3g} +\frac{1}{2} \inprodC{ f}{i \sigma_3\dotprod{\sigma}{t}\partial_{\bvec t} g} - \frac{\kappa}{2} \inprodC{f}{i \sigma_3\dotprod{\sigma}{n}g} \\
&= \frac{1}{2} \inprodC{f}{i \{ \dotprod{\sigma}{t} ,  \sigma_3 \} g} + \frac{\kappa}{2} \inprodC{f}{\dotprod{\sigma}{t}g}\\
&= \frac{\kappa}{2} \inprodC{f}{g}.
\end{align*}
Here, we have used $\partial_{\bvec t}  \bvec t_j = - \kappa \bvec n_j$ 
 to obtain the second line, the anti-commutation relations of the Pauli matrices and finally again the boundary conditions.
For any set $\mathcal O $ that does not contain corners, we find
\[
\inprod{ f}{i\sigma_3\partial_{\bvec t}  g}_{L^2(\bdry \cap \mathcal O)} = \frac{\kappa}{2} \inprod{f}{g}_{L^2(\bdry \cap \mathcal O)}.
\]
By density, this identity extends to all $f,g \in H^{1/2}(\bdry \cap \mathcal O)$ that satisfy boundary conditions, in particular, one can take $f = \gamma  u$ and $g = \gamma  v$, with $u,v\in\D(D^Q)$.
By dominated convergence, one can increase the set $\mathcal O$ to obtain the integral over all of $\bdry$ on both sides of the equality.
For \eqref{eq:q-form-D-L}, we sum \eqref{eq:q-form-D-L} with $\mathcal{O}= \Omega^+$ and $\mathcal{O}= \Omega^-$. Taking into account a change in sign of the tangent vector,
\begin{align*}
\inprod{D^L u }{D^L v}_{L^2(\Omega^+)\times L^2(\Omega^-)} =  \inprod{\nabla u }{\nabla v}_{L^2(\Omega^+)\times L^2(\Omega^-)} 
-i \int_{\bdry} \left( \inprodC{ u_+}{\sigma_3 \partial_{\bvec{t}} v_+}
 - \inprodC{ u_-}{\sigma_3 \partial_{\bvec{t}} v_-} 
\right)
\end{align*} 
Again, for functions
$f_+, g_+ \in C^1(\bdry)$, we define
\[
f_- = (\cosh(\alpha) + \sinh(\alpha) \dotprod{\sigma}{t}) f_+, \quad g_- = (\cosh(\alpha) + \sinh(\alpha) \dotprod{\sigma}{t}) g_+
\]
and obtain the desired form of the boundary term away from corners:
\begin{align*}
&-i \inprodC{ f_+}{\sigma_3 \partial_{\bvec{t}} g_+}
+i \inprodC{f_-}{\sigma_3 \partial_{\bvec{t}} g_-} \\
&=   -i \inprodC{f_+}{\sigma_3 \partial_{\bvec{t}} g_+} +i \inprodC{ (\cosh(\alpha) + \sinh(\alpha) \dotprod{\sigma}{t}) f_+}{\sigma_3 ((\cosh(\alpha) + \sinh(\alpha) \dotprod{\sigma}{t}))\partial_{\bvec{t}} g_+}  \\
& \qquad +i\sinh(\alpha) \inprodC{f_-}{\sigma_3 \dotprod{\sigma}{}\left(\partial_{\bvec{t}} \bvec{t} \right) g_+} \\
&= \kappa \sinh(\alpha) \inprodC{f_-}{ \dotprod{\sigma}{t} g_+}.
\end{align*}
As for the quantum-dot case, the result for all $\gamma u_\pm$ and $ \gamma v_\pm$ that satisfy boundary conditions, follows from here by density and dominated convergence. 
\end{proof}

We can now prove \Cref{lem:trace-ext}.

 \begin{proof}[Proof of \Cref{lem:trace-ext}]

We start noticing that, by the triangle inequality $H^1(\mathcal{O},\C^2)\subset\K(\mathcal{O})$ and there exists $C>0$ such that
\[
||u||_{\K(\mathcal{O})}\leq C||u||_{H^1(\mathcal{O})},\quad\text{for}\ u\in H^1(\mathcal{O},\C^2).
\]
Moreover, $H^1(\mathcal{O},\C^2)$ is dense in $\cK(\mathcal{O})$ with respect to the norm $||\cdot||_{\K(\mathcal{O})}$, see \cite[Proposition 2.12]{ourmieres2018strategy}).

  Fix a bounded extension operator (see \cite[Section 5]{adams2003sobolev})
  \[
  \xi : H^{1/2}(\partial \mathcal{O}, \C^2) \mapsto H^1(\mathcal{O}, \C^2).
  \]
  For $v \in \cK(\mathcal{O})$, we define $T_\xi v \in H^{-1/2}(\partial \mathcal{O}, \C^2)$ by 
  \begin{equation} \label{eq:trace-ext-def}
    T_\xi v [f] := i \inprod{v}{H\xi( f)}_{L^2(\mathcal{O})} -i \inprod{Hv}{\xi(f)}_{L^2(\mathcal{O})},\quad\text{for}\ f\in H^{1/2}(\partial\mathcal{O}, \C^2).
  \end{equation}
  Let us prove that $T_\xi$ has the necessary properties. Indeed
  \[
  \begin{split}
  |T_\xi v [f] &|\leq ||v||_{L^2(\mathcal{O})}||\xi(f)||_{\K(\mathcal{O})}+||Hv||_{L^2(\mathcal{O})}||\xi (f)||_{L^2(\mathcal{O})} \leq
C ||\xi (f)||_{H^1(\mathcal{O})}  ||v||_{\K(\mathcal{O})}\\
&\leq C_\xi||f||_{H^{1/2}(\de \mathcal{O})}  ||v||_{\K(\mathcal{O})},
  \end{split}
  \]
  for some $\eta_\xi >0$, and so $T_\xi $ is bounded from $\cK(\mathcal{O})$ to $H^{-1/2}(\partial \mathcal{O}, \C^2)$. 
  Moreover, by \eqref{eq:ibp}, it coincides with $\dotprod{\sigma}{n} \gamma$ on $H^{1}(\mathcal{O},\C^2)$.
  Since $H^1(\mathcal{O},\C^2)$ is dense in $\cK(\mathcal{O})$ (see \cite[Proposition 2.12]{ourmieres2018strategy}), 
   $\eta_E$ is independent of the choice of $E$ and it be denoted by $T$ to stress this. 
  
%
%
%
\end{proof}

\section{Properties of the angular operator}
\label{sec:appendix-angular}
In this appendix we prove some technical results about the \emph{angular operator}.

\begin{proof}[Proof of \Cref{lem:basis-eigen}]
Let us consider the quantum-dot case. 
The operator $D_{ang}^Q$ in symmetric on $\H^Q$. 
To prove the self-adjointness, we start noticing that by construction $\D((D^Q_{ang})^*)\subset H^1([0,\omega],\C^2)$. 
So, let $g\in H^1([0,\omega],\C^2)$, then integrating by parts we have that for any $f \in \D(D^Q_{ang})$
\[
\inprod{-i\sigma_3\partial_\theta f }{g}_{L^2}
-
\inprod{f}{-i\sigma_3\partial_\theta g}_{L^2}=
\left\langle -i\sigma_3 f(\omega), 
g(\omega)\right\rangle_{\C}
-
\left\langle-i \sigma_3 f(0), 
g(0)\right\rangle_{\C}.
\]
Since $f$ verifies the boundary conditions $f_2(0) =  f_1(0)$ and $f_2(\omega) = -  e^{i \omega} f_1(\omega)$, we have that
$g\in \D((D^Q_{ang})^*)$ if and only if 
\[
\inprod{f_1(\omega)}{g_1(\omega)+ e^{-i\omega}g_2(\omega)}_{\C}-
\inprodC{f_1(0)}{g_1(0)-g_2(0)}=0.
\]
Choosing firstly $f$ such that $f_1(\omega)=0$ and $f_1(0)\neq 0$, and secondly  $f$ such that $f_1(\omega)\neq 0$ and $f_1(0)= 0$, we deduce that $g$ has to verify the boundary conditions and so $g\in \D(D^Q_{ang})$.
The same proof can be adapted to prove the self-adjointness of $D^L_{ang}$.

Let us now find the eigenvalues of $D_{ang}$.
The generic solution of  $-i \sigma_3 \partial_\theta f=\lambda f$ is
\[
f(\theta)=
\begin{pmatrix}
c_1 e^{i\lambda\theta}\\
c_2 e^{-i\lambda\theta}
\end{pmatrix}.
\]
Let us impose the boundary conditions.
For the quantum-dot case, in order to satisfy the boundary conditions at $\theta= 0$, the eigenfunctions must take the form 
\[
f(\theta) = c_1 \begin{pmatrix} e^{i \lambda\theta} \\  e^{-i\lambda \theta}\end{pmatrix}.
\]
The boundary condition at $\theta = \omega$ implies that $\lambda =  \lambda_k = (2k+1) \frac{\pi}{2 \omega} -1/2 $, while  $c_1$ is determined by the normalization constant . 
To conclude, it is enough to observe that the operator $D^Q_{ang}$ is self-adjoint and it has compact resolvent, since $\D(D^Q_{ang})$ is compactly embedded in $\H^Q$. Thanks to this, and by the spectral theorem we can deduce that $\seq{f_k^Q}_{k\in\Z}$ is a basis of $\H^Q$.
Finally,
\[
\dotprod{\sigma}{e_r}f_k^Q(\theta)=
\frac{1}{\sqrt{2\omega}}
\begin{pmatrix}
0&e^{-i\theta}\\
e^{i\theta}&0
\end{pmatrix} \cdot
\begin{pmatrix}
 e^{i \lambda_k^Q \theta} \\  e^{-i \lambda_k^Q \theta}
\end{pmatrix}
=
\frac{1}{\sqrt{2\omega}}
\begin{pmatrix}
 e^{i \lambda_{-k-1}^Q \theta} \\  e^{-i \lambda_{-k-1}^Q \theta}
\end{pmatrix}
\]
where in the last equality we used that $-\lambda_k^Q-1=\lambda_{-k-1}^Q$.

Let us consider the Lorentz-scalar case.
Reasoning as before, in order to be an eigenfunction the the pair of spinors $(f_+ (\theta), f_-(\theta))$ has to be of the form
\[
f_{\pm} (\theta)= \begin{pmatrix}
a_{\pm} e^{i \lambda \theta} \\
b_{\pm} e^{-i \lambda \theta}
\end{pmatrix}.
\]
Let us impose the boundary conditions:
\begin{align}
\label{eq:bc0-2pi}
\begin{pmatrix}
e^{i2\pi\lambda}a_-\\
e^{-i2\pi\lambda}b_-
\end{pmatrix}&=
\begin{pmatrix}
\cosh(\alpha) & - \sinh(\alpha) \\ -\sinh(\alpha) & \cosh (\alpha)
\end{pmatrix} 
\cdot\begin{pmatrix}
a_+\\
b_+
\end{pmatrix},
\\
\label{eq:bc-omega}
\begin{pmatrix}
e^{i\omega\lambda}a_-\\
e^{-i\omega\lambda}b_-
\end{pmatrix}&=
\begin{pmatrix}
  \cosh(\alpha) & e^{-i \omega} \sinh(\alpha) \\ e^{i \omega} \sinh(\alpha) & \cosh (\alpha)
\end{pmatrix} 
\begin{pmatrix}
e^{i\omega\lambda}a_+\\
e^{-i\omega\lambda}b_+
\end{pmatrix}.
\end{align}
Combining \eqref{eq:bc0-2pi} and \eqref{eq:bc-omega},  after few trigonometric identities, we have that 
\begin{footnotesize}
\begin{equation}\label{eq:system-a+b+}
\begin{pmatrix}
2i\cosh (\alpha)e^{-i \lambda(\pi-\omega)}\cos(\pi(\lambda+1/2)) 
&
-2\sinh(\alpha)e^{-i(\omega/2+\pi\lambda)} \sin\left((\pi-\omega)(\lambda+1/2)\right)
\\
-2\sinh(\alpha)e^{i(\omega/2+\pi\lambda)} \sin\left((\pi-\omega)(\lambda+1/2)\right)
&
-2i\cosh (\alpha)e^{i \lambda(\pi-\omega)}\cos(\pi(\lambda+1/2)) 
\end{pmatrix}\cdot
\begin{pmatrix}
a_+\\
b_+
\end{pmatrix} =0.
\end{equation}
\end{footnotesize}
To get a non-trivial solution, the determinant of the matrix has to be zero, that is
\[
\cosh^2(\alpha) \cos^2(\pi(\lambda+1/2))-\sinh^2(\alpha) \sin^2\left((\pi-\omega)(\lambda+1/2)\right)=0.
\] 
We want to find the solutions of the following equations
\begin{align}
\label{eq:trascendental-lambda_k+}
&\cos(\pi (\lambda + 1/2))=\left|\tanh(\alpha)\right|\sin(|\pi -\omega| (\lambda+1/2)),\\
\label{eq:trascendental-lambda_k-}
&\cos(\pi (\lambda + 1/2))=-\left|\tanh(\alpha)\right|\sin(|\pi -\omega| (\lambda+1/2)).
\end{align}
By using standards tools, it is easy to see that both equations admit a countable family of solutions.
Let $\seq{\lambda_{2n}^L}_{n\in\Z}$ be the family of solutions of \eqref{eq:trascendental-lambda_k+} such that $\lambda^L_0$ is the unique solution in $(-1/2,0)$. Moreover let $\seq{\lambda_{2n+1}^L}_{n\in\Z}$ be the family of the solutions of \eqref{eq:trascendental-lambda_k-}
such that $\lambda_{-1}^L$ is the unique solution in $(-1,-1/2)$.
By construction, for any $k\in \Z$, $\lambda_k^L=-\lambda^L_{-k-1}-1$.
With this notation, $\lambda_0^L, \lambda^L_{-1}$ are the unique solutions to \eqref{eq:trascendental-lambda_k+} and \eqref{eq:trascendental-lambda_k-} that are in $(-1,0)$.
\begin{figure}[h!]
\begin{center}
\includegraphics[scale=0.17]{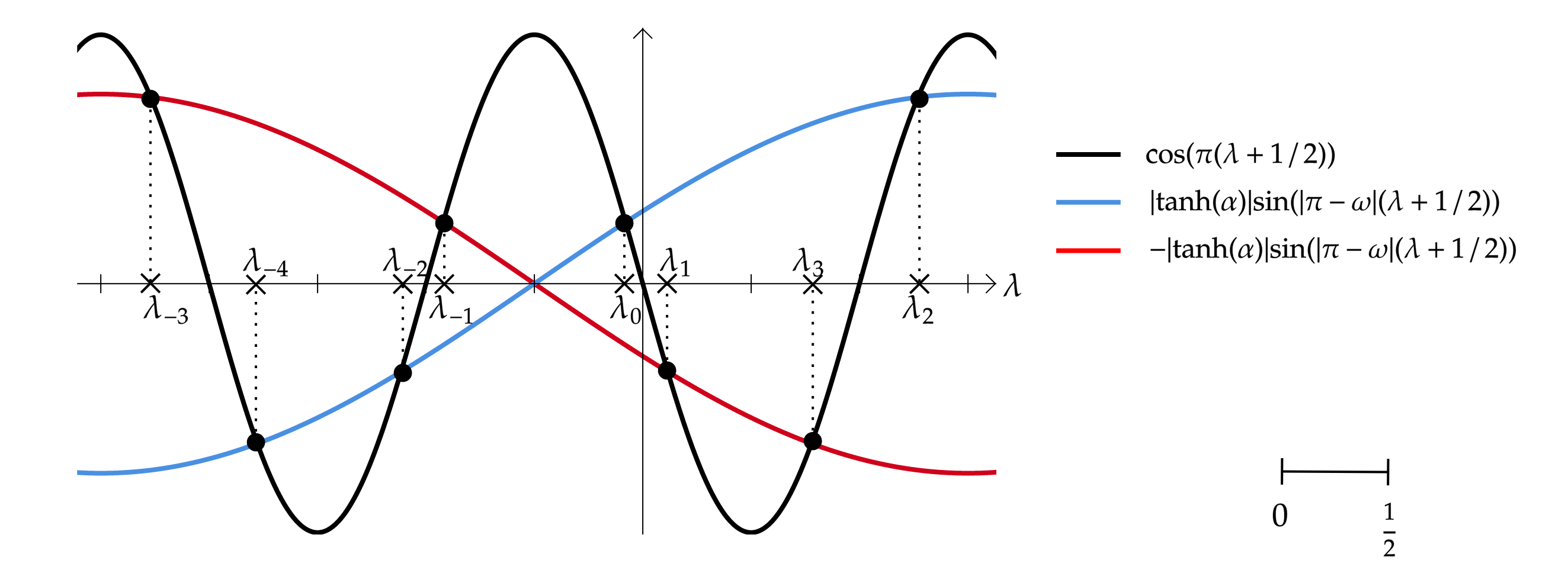} 
\caption{The solutions of \eqref{eq:trascendental-lambda_k+} and \eqref{eq:trascendental-lambda_k-} in $(-5/2,3/2)$, with $\alpha=1$ and $\omega=\pi/4$.}
\end{center}
\end{figure}

Assuming that \eqref{eq:trascendental-lambda_k+} and \eqref{eq:trascendental-lambda_k-} hold true, setting $\eta_k=\sign\left(\alpha\sin\left((\pi-\omega)(\lambda_k^L+1/2)\right)\right)$, we have that \eqref{eq:system-a+b+} is equivalent to
\begin{equation}
\begin{pmatrix}
ie^{i \lambda^L_k\omega}
&
-\eta_k e^{-i\omega/2}
\\
-\eta_k e^{i\omega/2}
&
-ie^{-i \lambda^L_k\omega}
\end{pmatrix}
\cdot
\begin{pmatrix}
a_+\\ b_+
\end{pmatrix}=0,
\end{equation}
whose solutions are
\begin{equation}\label{eq:def.a,b_+}
a_{k,+}:=  c_k \eta_k e^{-i\omega/2(\lambda_k^L+1/2)} ,\quad
b_{k,+}:=i c_k e^{i\omega/2(\lambda_k^L+1/2)},\qquad\text{with}\ c_k\in\C.
\end{equation}
Then, thanks to \eqref{eq:bc0-2pi} we have
\begin{equation}\label{eq:def.a,b_-}
\begin{split}
a_{k,-}&:= c_k e^{-i2\pi\lambda_k^L}\left(i\cosh(\alpha) e^{-i\omega/2(\lambda_k^L+1/2)}-\eta_k\sinh(\alpha) e^{i\omega/2(\lambda_k^L+1/2)}\right),\\
b_{k,-}&:=c_k e^{i2\pi\lambda_k^L}\left(\eta_k\cosh(\alpha) e^{i\omega/2(\lambda_k^L+1/2)}-i\sinh(\alpha) e^{-i\omega/2(\lambda_k^L+1/2)}\right).
\end{split}
\end{equation}
Finally, let us write $c_k=\rho_k e^{i\varphi_k}$, for $\rho_k>0$ and $\varphi_k\in[0,2\pi)$, and set
\begin{equation}\label{eq:def:rho_k}
\begin{split}
\rho_k:&=\left(|a_{k,+}|^2+|b_{k,+}|^2+|a_{k,-}|^2+|b_{k,-}|^2\right)^{-1/2}\\
&=\left[\cosh(\alpha)
	\left(\cosh(\alpha)+\sinh(\alpha)
		\eta_k^L\sin\left(\omega(\lambda_k^L+1/2)\right)\right)
		\right]^{-1/2};
\end{split}
\end{equation}
At this point, arguing as before, we conclude that $\seq{f_k^L}_{k\in\Z}$ is a basis of $\H^L$.
To conclude the proof, we only need to determine $\varphi_k$ in order to verify \eqref{eq:sigma.er.f_k}. Since $-\lambda_k^L-1=\lambda_{-k-1}$, we have that 
\[
\dotprod{\sigma}{e_r} f^L_{k,\pm}(\theta)=
\begin{pmatrix}
b_{k,\pm}e^{i\lambda^L_{-k-1}\theta}\\
a_{k,\pm} e^{-i\lambda^L_{-k-1}\theta}
\end{pmatrix}.
\]
Since $\eta_k=-\eta_{-k-1}$ and $\rho_k=\rho_{-k-1}$, we have that $a_{-k-1,\pm}=b_{k,\pm}$ if and only if 
$-\eta_k e ^{i\varphi_{-k-1}}=i e ^{i\varphi_{k}}$.
Since $\eta_k=e^{i(1-\eta_k)\pi/2}$ and $i=e^{i\pi/2}$ we can conclude the proof setting $\varphi_k=\eta_k/4$.
\end{proof}
\begin{proposition}\label{propr:Charge u_k}
Let $u_k$ be defined as in \Cref{def:f_k} and let $C$ be the charge conjugation operator defined in \eqref{eq:def.ChargeConj}. Then 
\begin{equation}\label{eq:Cu_k}
C u_k= u_k.
\end{equation}
\end{proposition}
\begin{proof}
We start computing
\[
C u_k(r,\theta)=
\phi(r)r^{\lambda_k}\sigma_1 \overline{f_k(\theta)},
\]
then it remains to prove $\sigma_1 \overline{f_k(\theta)}=f_k(\theta)$.
For the quantum-dot case, it is trivial. Let us consider the Lorentz scalar case. Then:
\[
\sigma_1 \overline{f_{k,\pm}(\theta)}=
\begin{pmatrix}
\overline{b_{k,\pm}}e^{i\lambda_k^L\theta}\\
\overline{a_{k,\pm}}e^{-i\lambda_k^L\theta}\\
\end{pmatrix},
\]
where $a_{k,\pm}$ and $b_{k,\pm}$ are defined in \eqref{eq:def.a,b_+} and \eqref{eq:def.a,b_-}, with $c_k=\rho_k e ^{i\eta_k/4}$ and $\rho_k>0$ is defined in \eqref{eq:def:rho_k}.
Arguing as above one can see that $\overline{a_{k,\pm}}=b_{k,\pm}$ and this concludes the proof.
\end{proof}

\section{Straightening of a curvilinear wedge}
\label{sec:coordinate-transf}
Throughout this section, we consider a domain that is bounded by a pair of semi-infinite curves of class $C^2$, intersecting at an angle $\omega$ at the origin.
We assume that the tangent and curvature have left and right limits at the origin.
Up to interchanging the interior and exterior, we can assume that $\omega \in (0, 2 \pi)$.  
Contrary to the previous convention, in this appendix, we take $y$-axis oriented along the bisector of $W_\omega$. We orient $\Omega$ in the same way, with the angle of opening $\omega$ at the origin. Then, we assume that $\Omega$ is bounded by a pair of semi-infinite curves that admit a parametrization $(x,c(x))$ for $x \ge 0$ and $(x,c(x))$, for $x \le 0$ respectively.
The border of $W_\omega$ is parametrized by $\left(x, \tfrac{\abs{x}}{\tan (\omega/2)}\right)$. 

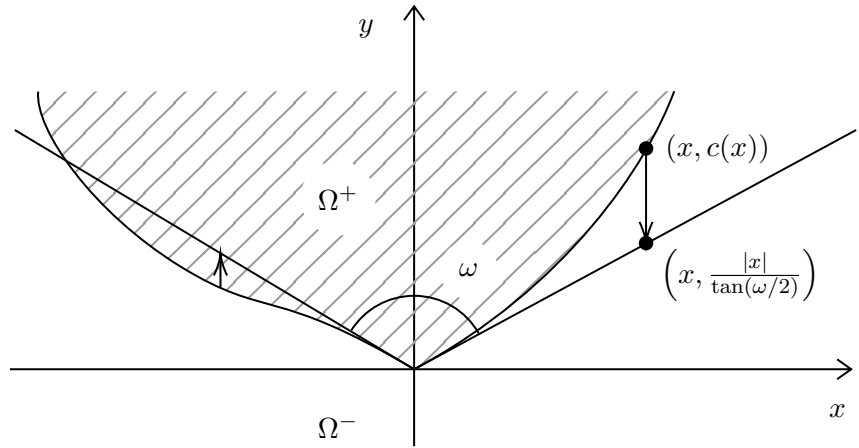
\begin{figure}[H]
\begin{centering}

 
\tikzset{
pattern size/.store in=\mcSize, 
pattern size = 5pt,
pattern thickness/.store in=\mcThickness, 
pattern thickness = 0.3pt,
pattern radius/.store in=\mcRadius, 
pattern radius = 1pt}
\makeatletter
\pgfutil@ifundefined{pgf@pattern@name@_qf0uwwjre}{
\pgfdeclarepatternformonly[\mcThickness,\mcSize]{_qf0uwwjre}
{\pgfqpoint{0pt}{0pt}}
{\pgfpoint{\mcSize+\mcThickness}{\mcSize+\mcThickness}}
{\pgfpoint{\mcSize}{\mcSize}}
{
\pgfsetcolor{\tikz@pattern@color}
\pgfsetlinewidth{\mcThickness}
\pgfpathmoveto{\pgfqpoint{0pt}{0pt}}
\pgfpathlineto{\pgfpoint{\mcSize+\mcThickness}{\mcSize+\mcThickness}}
\pgfusepath{stroke}
}}
\makeatother
\tikzset{every picture/.style={line width=0.75pt}} 
 \usetikzlibrary{patterns}

\begin{tikzpicture}[x=0.75pt,y=0.75pt,yscale=-1,xscale=1]

\draw  (137.44,220.27) -- (557.35,220.27)(339.19,37.7) -- (339.19,258.97) (550.35,215.27) -- (557.35,220.27) -- (550.35,225.27) (334.19,44.7) -- (339.19,37.7) -- (344.19,44.7)  ;
\draw [pattern=_qf0uwwjre,pattern size=11.25pt,pattern thickness=0.75pt,pattern radius=0pt, pattern color={rgb, 255:red, 155; green, 155; blue, 155}]   (151.55,80.6) .. controls (149.55,99.6) and (178.97,141.59) .. (222.55,168.6) .. controls (266.13,195.61) and (268.46,177.94) .. (339.19,220.27) .. controls (415.12,179.27) and (455.48,117.36) .. (469.02,80.41) ;

\draw    (139.88,99.93) -- (339.19,220.27) ;

\draw    (339.19,220.27) -- (559.83,99.93) ;

\draw    (454.96,109.23) -- (454.96,155.1) ;
\draw [shift={(454.96,157.1)}, rotate = 270] [color={rgb, 255:red, 0; green, 0; blue, 0 }  ][line width=0.75]    (10.93,-3.29) .. controls (6.95,-1.4) and (3.31,-0.3) .. (0,0) .. controls (3.31,0.3) and (6.95,1.4) .. (10.93,3.29)   ;

\draw    (242.53,179.45) -- (242.53,164.17) ;
\draw [shift={(242.53,162.17)}, rotate = 450] [color={rgb, 255:red, 0; green, 0; blue, 0 }  ][line width=0.75]    (10.93,-3.29) .. controls (6.95,-1.4) and (3.31,-0.3) .. (0,0) .. controls (3.31,0.3) and (6.95,1.4) .. (10.93,3.29)   ;

\draw  [draw opacity=0] (307.86,200.69) .. controls (314.38,190.26) and (325.98,183.32) .. (339.19,183.32) .. controls (353,183.32) and (365.04,190.9) .. (371.38,202.13) -- (339.19,220.27) -- cycle ; \draw   (307.86,200.69) .. controls (314.38,190.26) and (325.98,183.32) .. (339.19,183.32) .. controls (353,183.32) and (365.04,190.9) .. (371.38,202.13) ;
\draw  [fill={rgb, 255:red, 0; green, 0; blue, 0 }  ,fill opacity=1 ] (451.84,109.23) .. controls (451.84,107.5) and (453.24,106.1) .. (454.96,106.1) .. controls (456.69,106.1) and (458.09,107.5) .. (458.09,109.23) .. controls (458.09,110.95) and (456.69,112.35) .. (454.96,112.35) .. controls (453.24,112.35) and (451.84,110.95) .. (451.84,109.23) -- cycle ;
\draw  [fill={rgb, 255:red, 0; green, 0; blue, 0 }  ,fill opacity=1 ] (451.84,156.83) .. controls (451.84,155.1) and (453.24,153.7) .. (454.96,153.7) .. controls (456.69,153.7) and (458.09,155.1) .. (458.09,156.83) .. controls (458.09,158.55) and (456.69,159.95) .. (454.96,159.95) .. controls (453.24,159.95) and (451.84,158.55) .. (451.84,156.83) -- cycle ;

\draw (490.74,109.23) node   {$( x,c( x))$};
\draw (501.96,173.09)node   {$\left( x,\frac{|x|}{\tan( \omega /2)}\right)$};
\draw (550.55,241.11) node   {$x$};
\draw (315.35,49.1) node   {$y$};
\draw  [draw opacity=0][fill={rgb, 255:red, 255; green, 255; blue, 255 }  ,fill opacity=1 ]  (366.5, 170.43) circle [x radius= 14.6, y radius= 14.6]   ;
\draw (366.5,170.43) node   {$\omega $};
\draw  [draw opacity=0][fill={rgb, 255:red, 255; green, 255; blue, 255 }  ,fill opacity=1 ]  (301.19, 134.22) circle [x radius= 18.77, y radius= 18.77]   ;
\draw (301.19,134.22) node   {$\Omega ^{+}$};
\draw  [draw opacity=0][fill={rgb, 255:red, 255; green, 255; blue, 255 }  ,fill opacity=1 ]  (301.19, 248.74) circle [x radius= 18.4, y radius= 18.4]   ;
\draw (301.19,248.74) node   {$\Omega ^{-}$};

\end{tikzpicture}
\caption{The domain $\Omega$ and the wedge $W_\omega$.}
\end{centering}
\end{figure}

Consider the coordinate transformation $S:(x, y ) \in \Omega\mapsto S(x,y) \in W_\omega$ defined by
\[
S(x,y) = \left(x, y - c(x) + \frac{\abs{x }}{ \tan(\omega/2)} \right).
\] 
Since the boundary of $\Omega$ is $C^2$ except at the origin, where it is tangent to the wedge, 
\[
 \abs*{ c'(x) - \frac{\sign (x )}{\tan (\omega/2)}} \le \abs{x} \sup_{\R \setminus \{0\}} \abs{c''(x)}. 
\] 
The Jacobian matrix of $S$ is
\[
J(x, y) = 	\begin{pmatrix}  
			1 & 0 \\
 			- c'(x) + \frac{\sign (x )}{\tan(\omega/2)} & 1	
		\end{pmatrix}
 = \id  + O(\abs{x}).
\]
The relative angle of the rotation of the boundary tangent is
\[
\delta(x ) 
\equiv \sign (x) \arctan (1/c'(x)) - \omega /2)
\quad\text{and}\quad
|\delta(x)| \sim \abs{x}.
\]
Now, for $u = (u_+, u_-) \in L^2(W_\omega)\times L^2(W_\omega^c) $, we define $U_S u  \in L^2(W_\omega)\times L^2(W_\omega^c) $ by
\[ 
(U_S u)_\pm(x,y) :=
e^{i\delta(x)\sigma_3/2}u_\pm (S(x,y))=
\begin{pmatrix}
e^{i\delta(x)/2}&0\\
0&e^{-i\delta(x)/2}
\end{pmatrix} u_\pm (S(x,y)). 
\] 
One checks that
\[
e^{- i \delta \sigma_3} \begin{pmatrix} 0 & e^{-i \omega/2} \\ e^{i \omega/2} & 0 \end{pmatrix} e^{ i \delta \sigma_3} = \begin{pmatrix} 0 & e^{-i(\omega/2 + \delta)} \\ e^{i (\omega/2 + \delta)} & 0 \end{pmatrix}.
\]
If $u_\pm $ have boundary traces that satisfy Lorentz-scalar boundary conditions at the boundary of $\bdry$, we have that
\[\begin{split}
& \left(- i \bvec \sigma \cdot \bvec n_{W_\omega} + \mu \sigma_3 \right)  \gamma (U_S u)_+ 
-\left(- i \bvec \sigma \cdot \bvec n_{W_\omega} - \mu \sigma_3 \right) \gamma (U_S u)_-  \\
&\,= e^{ i \delta(x)  \sigma_3 / 2} \left(  \left(- i \bvec \sigma \cdot \bvec n_{\Omega} + \mu \sigma_3 \right)  \gamma u_+(S(x,y)) 
-\left(- i \bvec \sigma \cdot \bvec n_{\Omega} - \mu \sigma_3 \right) \gamma  u_- (S(x,y)) \right) =0.
\end{split}\]
We can also compute
\[\begin{split}
H U_S( u) =& e^{ i \delta(x)  \sigma_3} U_S(H u) -i \left(- c'(x) + \frac{\sign (x)}{\tan(\omega/2)} \right) \sigma_1 e^{ i \delta(x)  \sigma_3/2} \partial_y u  \circ S \\
&+ (\delta'(x)/2 ) \sigma_1e^{ i \delta(x)  \sigma_3/2} .
\end{split}\]
Then, expanding the first exponential around $x =0$, we have
\[
U_S^* H U_S=
H+
(e^{ i \delta(x)  \sigma_3}-1) H
-i \left(- c'(x) + \frac{\sign (x)}{\tan(\omega/2)} \right) \sigma_1 e^{ i \delta(x)  \sigma_3} \partial_y
+ (\delta'(x)/2 ) \sigma_1e^{ i \delta(x)  \sigma_3}.
\]
We recover \eqref{eq:def_perturbation}, setting
\begin{equation}\label{eq:def-L-M}
\begin{split}
L_1&:=-i(e^{ i \delta(x)  \sigma_3}-1)\sigma_1,\\
L_2&:=-i(e^{ i \delta(x)  \sigma_3}-1)\sigma_2-i \left(- c'(x) + \frac{\sign (x)}{\tan(\omega/2)} \right) \sigma_1 e^{ i \delta(x)  \sigma_3} ,\\
M&:=(\delta'(x)/2 ) \sigma_1e^{ i \delta(x)  \sigma_3}.
\end{split}
\end{equation}

\section*{Acknowledgements}
We would like to thank Luis Vega for the enlightening discussions. 
This work was partially developed while F.~P.~was employed at \emph{BCAM - Basque Center for Applied Mathematics}, 
and he was supported by
ERCEA Advanced Grant 2014 669689 - HADE, by the MINECO project MTM2014-53850-P, by
Basque Government project IT-641-13 and also by the Basque Government through the BERC
2018-2021 program and by Spanish Ministry of Economy and Competitiveness MINECO: BCAM
Severo Ochoa excellence accreditation SEV-2017-0718. 
He has also has received funding from
the European Research Council (ERC) under the European Union’s Horizon 2020
research and innovation programme (grant agreement MDFT No 725528 of Mathieu Lewin).
The work of H.~VDB.~has been partially supported by CONICYT (Chile) through PCI Project REDI170157, Fondecyt Projects \# 318--0059 and \# 118--0355, and Grant PIA AFB-170001.

\end{document}